\documentclass{amsart}
\usepackage{amsmath}
\usepackage{amsfonts}
\usepackage{amssymb}
\usepackage{arydshln}
\usepackage{graphicx}%
\newtheorem{theorem}{Theorem}[section]
\newtheorem{resonance-theorem}[theorem]{Resonance Theorem}

\newtheorem{corollary}[theorem]{Corollary}

\newtheorem{lemma}[theorem]{Lemma}

\newtheorem{proposition}[theorem]{Proposition}
\newtheorem{remark}[theorem]{Remark}

\newcommand{\cri}{{\rm \mathsf{cr}}}
\newcommand{\ind}{{\rm ind}}
\newcommand{\nullity}{{\rm null}}
\newcommand{\nulli}{{\rm null}}

\newcommand{\cg}{\gamma}
\newcommand{\malpha}{\ol{\alpha}}
\newcommand{\ol}{\overline}
\newcommand{\R}{\mathbb{R}}
\newcommand{\beq}{\begin{equation}}
\newcommand{\eeq}{\end{equation}}
\newcommand{\beqn}{\begin{equation*}}
\newcommand{\eeqn}{\end{equation*}}
\newcommand{\group}{S}
\title{Resonance for loop homology of spheres}
\author{Nancy Hingston}
\address{Department of Mathematics and Statistics, College of New Jersey,
Ewing, New Jersey 08628, USA}
\email{hingston@tcnj.edu}
\author{Hans-Bert Rademacher}
\address{Mathematisches Institut, Universit{\"a}t Leipzig,
04081 Leipzig, Germany}
\email{rademacher@math.uni-leipzig.de}
\date{2011-05-04, revised: 2012-04-04}
\begin{document}
\begin{abstract}
A Riemannian or Finsler metric on a compact manifold $M$ gives rise to a
length function on the free loop space $\Lambda M$, whose critical points are
the closed geodesics in the given metric. If $X$ is a homology class on
$\Lambda M$, the \textquotedblleft minimax\textquotedblright\ critical level
$\cri(X)$ is a critical value. Let $M$ be a sphere of dimension
$>2$, and fix a metric $g$ and a coefficient field $G$. We prove that the limit
as $\deg(X)$ goes to infinity of $\cri(X)/\deg(X)$ exists. We call this limit
$\overline{\alpha}=\overline{\alpha}(M,g,G)$ the \textit{global mean
frequency} of $M.$ As a consequence we derive resonance statements
for closed geodesics on spheres; in particular either all homology on
$\Lambda$ of sufficiently high degreee lies hanging on closed geodesics of
mean frequency  (length/average index) $\ \overline{\alpha}$, or there is a
sequence of infinitely many closed geodesics whose mean frequencies converge
to $\overline{\alpha}$. \ The proof uses the Chas-Sullivan product and results
of Goresky-Hingston~\cite{GH}.
\end{abstract}
\keywords{closed geodesics, average index, mean frequency,
Chas-Sullivan product, free loop space, critical value of
homology classes}
\subjclass[2000]{53C22;58E10}
\thanks{Part of this work was carried out during the
inspiring
joint workshop on Geodesics sponsored by the American Institute of Mathematics
(AIM) and the Chern Institute of Mathematics (CIM),August 23-27, 2010,
at the Chern Institute of Mathematics, Nankai University, Tianjin China.
The first named author is grateful fot the support of the Institute for Advanced Study,
Princeton.
The second named author acknowledges the support by the
Deutsche Forschungsgemeinschaft (SPP 1154).\\
We are grateful to the referee for his suggestions}
\maketitle
\tableofcontents
\section{Introduction and Statement of Results}
\label{sec:introduction} 
Fix a Riemannian metric $g$ or a Finsler metric $f$
on a compact manifold $M.$ Hence the corresponding norm is given by $\Vert
v\Vert ^{2}=g(v,v)$ resp. $\Vert v\Vert =f(v).$ As common notation we use
the letter $g$ also for a Finsler metric. 
Let $\Lambda $ be the space of $H^{1}$maps $\gamma :S^{1}=\mathbb{R}/\mathbb{%
Z}\rightarrow M$, and let $F:\Lambda \rightarrow \mathbb{R}$ be the square
root of the energy function: 
\begin{equation*}
F(\gamma )=\left( \int_{0}^{1}\Vert \gamma ^{\prime }(t)\Vert ^{2}\right)
^{1/2}
\end{equation*}%
Then $F$ and the length function $l(\gamma )=\int_{0}^{1}\Vert \gamma
^{\prime }(t)\Vert \,dt$ agree on loops that are parameterized proportional
to arclength. Fix a coefficient group $G$. The \emph{critical level} of a
homology class $\ X\in H_{\ast }(\Lambda ;G)$ is defined to be%
\begin{eqnarray*}
\cri(X) &=&\cri_{g,G}(X)=\inf \left\{ a\in \mathbb{R}:\ X\text{
is supported on }\Lambda ^{\leq a}\right\} \\
&=&\inf \left\{ a\in \mathbb{R}:\ X\text{ is in the image of }H_{\ast }(%
\text{ }\Lambda ^{\leq a})\right\} ,
\end{eqnarray*}%
where $\Lambda ^{\leq a}\subset \Lambda $ is the space of loops $\gamma $
with $F(\gamma )$ $\leq a.$\ The critical level of a cohomology class $x\in
H^{\ast }(\Lambda )$ is defined to be 
\begin{eqnarray*}
\cri (x) &=&\cri_{g,G}(x)=\sup \left\{ a\in \mathbb{R}:\ x\text{
is supported on }\Lambda ^{\geq a}\right\} \\
&=&\sup \left\{ a\in \mathbb{R}:\ x\text{ is in the image of }H^{\ast
}(\Lambda ,\text{ }\Lambda ^{<a})\right\} .
\end{eqnarray*}
These are critical values of the function $F$ (unless $x=0\in H^{\ast
}(\Lambda )$, in which case $\cri(x)=\infty $). Since the critical
points of $F$ are closed geodesics on $M$, it follows that for each
nontrivial homology or cohomology class $X$, there is a closed geodesic of
length ${\cri}(X)$. 
\begin{theorem}
\label{thm:resonance} \,\textrm{\textsc{(Resonance Theorem)}} \ A Riemannian
or Finsler metric $g$ on $S^n, n>2$, and a field $G$ determine a \emph{%
global mean frequency } $\overline{\alpha}=:\overline{\alpha}_{g,G}>0$ with
the property that 
\begin{equation*}
\deg(X)-\overline{\alpha}\,\cri(X)
\end{equation*}
is bounded as $X$ ranges over all nontrivial homology and cohomology classes
on $\Lambda$ with coefficients in $G$.\newline
Thus the countably infinite set of points $\left({\cri}%
(X),\deg(X)\right)$ in the $(\ell,d)-$plane lies in bounded distance from
the line $d=\overline{\alpha}\ell.$
\end{theorem}
We will derive an explicit bound for $\left\vert \deg (X)-\overline{\alpha }%
\,\cri(X)\right\vert $ in Section~\ref{sec:proof}. If the metric $g$
carries only finitely many closed geodesics, then for $\deg (X)$
sufficiently large, each homology class $X$ satisfies: 
\begin{equation*}
|\deg (X)-\overline{\alpha }\,\cri(X)|\leq n\,
\end{equation*}%
by Lemma~\ref{lem:X}. 
If $\gamma$ is a closed geodesic, then so is each of its \textit{iterates} $%
\gamma^{m}$: $\gamma^{m}(t)=\gamma(mt)$. The \textit{index} of $\gamma$ is
the index of the hessian of the functional $F$ on the free loop space, and
the \textit{average index }is 
$\alpha_{\gamma}=\lim_{m\rightarrow\infty}\ind(\gamma^{m})/m$. 
\begin{theorem}
\label{thm:density} \textrm{\textsc{(Density Theorem) }}\thinspace\ Let $%
\overline{\alpha }_{g,\mathbb{Q}}$ \ be the rational global mean frequncy \
of a Riemannian or Finsler metric $g$ on $S^{n}$, $n>2$. For any $%
\varepsilon >0$ we have the following estimate (compare Rademacher~\cite[%
Thm.1.2]{Rad95}) for the sum of inverted average indices $\alpha _{\gamma }$
of geodesics on $(S^{n},g):$ 
\begin{equation*}
\underset{\gamma }{{\textstyle\sum }}\,\frac{1}{\alpha _{\gamma }}\geq
\left\{ 
\begin{array}{ccc}
\frac{1}{n-1} & ; & n\mbox{ \rm odd } \\ 
\frac{1}{2(n-1)} & ; & n\mbox{ \rm  even}%
\end{array}%
\right.
\end{equation*}%
where we sum over a maximal set of prime, geometrically distinct closed
geodesics $\gamma $ whose mean frequency $\overline{\alpha }_{\gamma
}=:\alpha _{\gamma }/\ell (\gamma )$ satisfies: $\overline{\alpha }_{\gamma
}\in \left( \overline{\alpha }_{g,\mathbb{Q}}-\varepsilon ,\overline{\alpha }%
_{g,\mathbb{Q}}+\varepsilon \right) .$
\end{theorem}
For each closed geodesic $\gamma$, 
\begin{equation*}
\lim_{d\to \infty}
\frac{\#\{m\ge 1\,;\, \ind(\gamma^m)\le d\}}{d}=\frac{1}{\alpha_{\gamma}}.
\end{equation*}
Thus $1/\alpha_{\gamma}$ represents the density, in units of geometrically
distinct closed geodesics per degree, of all the iterates of $\gamma$. \
Since $\gamma$ and all its iterates have the same mean frequency, the sum in
the \textsc{Density Theorem }~\ref{thm:density} represents the total density, in these units,
of all the iterates of all the closed geodesics whose mean frequency lies in
the given interval. 
\begin{corollary}
\label{cor:pinch} Let $n$ be odd. Let $g$ be a Riemannian metric on the $n$%
-sphere $S^{n}$ whose sectional curvature $K$ satisfies 
\begin{equation*}
\frac{1}{4}<K\leq1
\end{equation*}
or let $g$ be a non-reversible Finsler metric with reversibility $\lambda>1$
(cf.~\cite{Rad04}) whose flag curvature $K$ satisfies: 
\begin{equation*}
\left( 1-\frac{1}{\lambda+1}\right) ^{2}<K\leq1.
\end{equation*}
Then at least one of the following holds:
\begin{itemize}
\item[(i)] There are at least two closed geodesics $\gamma$ with mean
frequency $\overline{\alpha}_{\gamma}$ equal to $\overline{\alpha}_{g,%
\mathbb{Q}}.$
\item[(ii)] There is a sequence of prime closed geodesics $\{\gamma_{j}\},$
with mean frequencies $\overline{\alpha}_{j}\neq\overline {\alpha}_{g,%
\mathbb{Q}}$ satisfying 
\begin{equation*}
\underset{j\rightarrow\infty}{\lim}\overline{\alpha}_{j}= \overline{\alpha}%
_{g,\mathbb{Q}}.
\end{equation*}
\end{itemize}
\end{corollary}
Using a Killing field $V$ on the standard sphere one can define an 
one parameter family $f_{\varepsilon}$, $\varepsilon\in \lbrack0,1)$ of
Finsler metrics of constant flag curvature on the sphere $S^{n}$ with the
following properties: For $\varepsilon=0$ the metric is the standard
Riemannian metric. For $\varepsilon\in(0,1)$ the metric is a non-reversible
Finsler metric of reversibility $\lambda=(1+\varepsilon)/(1-\varepsilon)$.
The geodesic flow is the composition of the geodesic flow of the standard
metric and the flow of the Killing field $\varepsilon V.$ For $\varepsilon$
irrational there are only finitely many closed geodesics, all of which are
nondegenerate. These metrics were first introduced by Katok, cf.~\cite{Zi}, 
\cite[thm.5]{Rad95}. The number of geometrically distinct distinct closed
geodesics is\ $n$ if $n$ is even and $n+1$ if $n$ is odd. In the following
Corollary we present properties of metrics nearby a Katok metric on $S^n:$ 
\begin{corollary}
\label{cor:Katok} Let $n$ be odd. \ Let $S^{n}$ carry a Katok metric $g_{0}$
of constant flag curvature $1$ and let $U\subset\Lambda$ be a neighborhood
of the set of closed geodesics on $\left(S^{n},g_{0}\right).$ Let $N\in%
\mathbb{Z}.$ There is a neighborhood $\mathcal{W}$ of $g_{0}$ in the space
of metrics so that for every $g\in\mathcal{W},$ at least one of the
following is true:
\begin{itemize}
\item[(i)] \ There are at least two closed geodesics $\gamma \subset U$ with
mean frequency $\overline{\alpha}_{\gamma}$ equal to $\overline{\alpha}_{g,%
\mathbb{Q}}.$
\item[(ii)] \ There is a sequence of prime closed geodesics $\{\gamma
_{j}\}, $ with mean frequencies 
$\overline{\alpha }_{j}\neq \overline{\alpha}_{g,\mathbb{Q}}$ satisfying 
\begin{equation*}
\underset{j\rightarrow \infty }{\lim }\overline{\alpha }_{j}=\overline{%
\alpha }_{g,\mathbb{Q}}.
\end{equation*}
\item[(iii)] \ There are at least $N$ closed geodesics with mean frequency
equal to $\overline{\alpha}_{g,\mathbb{Q}}.$
\end{itemize}
\end{corollary}
In the Katok metrics the set of prime closed geodesics is compact, and all
closed geodesics $\gamma$ have the same mean frequency
$\overline{\alpha}_{\gamma}=\overline{\alpha}_{g_{0},\mathbb{Q}}$. 
Thus the Katok metrics are 
\textit{perfectly resonant metrics}. Corollary~\ref{cor:Katok} shows that it
is difficult to pry the mean frequencies apart! But we will also prove the
following:
\begin{theorem}
\textrm{\textsc{(Open mapping theorem)}} \label{thm:open_mapping_theorem}
Let $M$ be a compact manifold. Let $\mathcal{G}=\mathcal{G}^{r}(M)$ be the
set of $C^{r} $ Riemannian or Finsler metrics on $M$, with $r\geq 2$ for a
Riemannian metric and $r\ge 4$ for a Finsler metric. Let 
\begin{align*}
\mathcal{G}_{j} =\{(g,\gamma_{1},\gamma_{2},..,\gamma_{j}):g\in \mathcal{G}%
\text{ and } \text{each }&\gamma_{i}\text{ is a geodesic} \\
& \text{in the metric }g\} \subset\mathcal{G}\times\Lambda^{j}
\end{align*}
The map 
\begin{equation*}
\Phi_{j}:\mathcal{G}_{j}\longrightarrow \mathbb{R}^{j}
\end{equation*}
by $\Phi_{j}(g,\gamma_{1},\gamma_{2},..,\gamma_{j})=(\overline {\alpha}_{1},%
\overline{\alpha}_{2},...\overline{\alpha}_{j})$, where $\overline{\alpha}%
_{i}$ is the mean frequency of the geodesic $\gamma_{i}$ in the metric $g$,
is an open mapping at each point $(g,\gamma_{1},\gamma_{2},..,\gamma_{j})$
where the $\gamma_{i}$ are geometrically distinct and of positive length.
That is, at each such point the image of any open set containing the point $%
(g,\gamma_{1},\gamma_{2},..,\gamma_{j})$ contains a neighborhood of $%
\Phi_{j}(g,\gamma_{1},\gamma_{2},..,\gamma_{j}).$
\end{theorem}

This theorem supports our intuition that the mean frequencies of
geometrically distinct closed geodesics can be perturbed independently. The
argument of Brian White in \cite{W} \ shows that the space $\mathcal{G}_{j}$
is a Banach manifold, and that the projection $\mathcal{G}_{j}\rightarrow 
\mathcal{G}$ is a smooth map of Fredholm index $0$.


\begin{lemma}
\label{lem:continuity} \textrm{\textsc{(Continuity) }} Fix a field $G$. Let $%
M=S^{n}$ and $r\geq 2$. \ The \ map $\mathcal{G}\rightarrow \mathbb{R}$
given by $g\mapsto \overline{\alpha }_{g,G}$ is continuous.
\end{lemma}

If a metric $g$ on a manifold $M$ has global mean frequency $\overline{%
\alpha }$, the scaled metric $s^{2}g$ (where all lengths are scaled by a
factor $s$) will have mean frequency $\overline{\alpha }/s$. \ But we do not
know how to perturb $\overline{\alpha }$ keeping the volume fixed, or
keeping the mean frequency of a closed geodesic fixed. \ 

In contrast to the Katok metrics, \ the standard ellipsoid metrics are
highly nonresonant. \ The standard $n$-dimensional ellipsoid depends on $n+1$
parameters, and has $(n^{2}+n)/2$ "short" closed geodesics, namely the
intersection of the ellipsoid with any of the coordinate planes. \ \ We
prove the following:

\begin{theorem}
\label{thm:ellipsoid} \textrm{\textsc{(Ellipsoid Theorem)}}

\begin{itemize}
\item[(a)] There is a nonempty open set of parameters $(a_{0},a_{1},a_{2})$
for which the mean frequencies of the three short closed geodesics on the $2$%
-dimensional ellipsoid $M\left(a_0,a_1,a_2\right)$ are distinct.

\item[(b)] For $n\geq 3$ there is a nonempty open set of parameters $%
(a_{0},a_{1,}...a_{n})$ for which there are at least $(n+1)/2$ different
mean frequencies among the short closed geodesics on the $n$-dimensional
ellipsoid $M\left( a_{0},a_{1},\ldots ,a_{n}\right) \,.$%
\end{itemize}
\end{theorem}

The Density Theorem may not be optimal. If we could prove a resonance
theorem for equivariant homology, we could improve the minimum density in
Theorem~\ref{thm:density} to 
\begin{equation*}
\underset{\gamma }{{\textstyle\sum }}\,\frac{1}{\alpha _{\gamma }}\geq \frac{%
1}{2}
\end{equation*}%
or, in the nondegenerate case, to 
\begin{equation*}
\underset{\gamma }{{\textstyle\sum }}\,\frac{1}{\alpha _{\gamma }}\geq
\left\{ 
\begin{array}{ccc}
\frac{n}{2(n-1)} & ; & \text{ \ \ }n\text{ even} \\ 
\frac{n+1}{2(n-1)} & ; & \text{ \ \ }n\text{ odd.}%
\end{array}%
\right.
\end{equation*}%
where in each case the sum is over closed geodesics whose mean frequency
lies in a given $\varepsilon $-neighborhood of $\overline{\alpha }$. The
minimum density $\frac{1}{2}$ was obtained by Rademacher, see~\cite{Rad95}, but
the sum there was taken in general over closed geodesics whose mean
frequency lies in a given $\varepsilon $-neighborhood of an interval. With
these improved densities, the number of resonant closed geodesics ($2$) in
Corollary 1.3 (i) and Corollary 1.4 (i) would be replaced by $(n-1)/2$ and,
in the nondegenerate case, by $(n+1)/2$. This last estimate is within a
factor of $2$ of the optimal number, as the odd-dimensional Katok metrics
have as few as $n+1$ closed geodesics. Using the {\em Common Index Jump Theorem}
of Long and Zhu~\cite{LZ} and recent results of Wang~\cite{WW}
one might do even better. 
However products play a crucial role in our proof; 
at present there is not a theory
incorporating equivariant homology/cohomology and the products 
$\circledast$ and $\bullet$ (cf.~Equation~\ref{eq:products})
we have used to prove the Resonance Theorem~\ref{thm:resonance}.
There are related
products on equivariant homology and cohomology 
(see~\cite[p.19]{CS}, and\cite[17.4, p.156]{GH}), but they vanish for spheres,
see Proposition~\ref{pro:vanishing}.

Using a version of the Fadell-Rabinowitz index~\cite{FR} and following
similar results by Ekeland-Hofer~\cite{EH} for periodic orbits of of convex
Hamiltonian energy hypersurfaces in $\mathbb{R}^{2n}$ Rademacher studied in ~%
\cite{Rad95} limits closely related to 
\begin{equation*}
\lim_{m\rightarrow \infty }\frac{\cri(z\cup \eta ^{\cup m})%
}{m}
\end{equation*}%
for a sequence of equivariant cohomology classes $z\cup \eta ^{\cup m}\in
H_{S\mathbb{O}(2)}^{\ast }\left( \Lambda ,\Lambda ^{0};\mathbb{Q}\right) $
where $\eta $ is a nonnilpotent element for the cup product $\cup $ in
equivariant cohomology. It is not known if there is a metric for which these
limits do not exist. However if 
\begin{equation}
\limsup_{m \to\infty }\frac{\cri\left(z\cup \eta
^{\cup m}\right)}{m}\neq \liminf_{n\to \infty } 
\frac{\cri \left(z\cup \eta ^{\cup m}\right)}{m}  \label{eq:limsup}
\end{equation}%
for a fixed metric $g$ on $M$, there would be infinitely many closed
geodesics on $M$ for any metric in a neighborhood of $g,$ 
cf.~\cite[Cor.6.5]{Rad95}. 
In Section~\ref{sec:critical} we will look at analogous 
limits for the loop products, and prove
an analogous theorem. Proposition~\ref{pro:muX}, Corollary~\ref{cor:muX}, and the 
Theorem~\ref{thm:interval}
({\em Interval Theorem}) should be compared to 
Ekeland-Hofer~\cite{EH} and Rademacher~\cite{Rad95}.
If one could extend the resonance theorem to equivariant
homology/cohomology, this would imply that the index interval discussed in~%
\cite{Rad95} is always a point, and that Inequality~\ref{eq:limsup} never
holds.

\smallskip

\textsc{Organization of the paper }

\smallskip

After discussing the Resonance Theorem in Section~\ref{sec:discussion}
we present general facts about critical
levels, duality, and coefficients in Section~\ref{sec:critical}. 
Section~\ref{sec:mean} is concerned with
limits of the type given in Equation~\ref{eq:limsup} for the loop homology and loop
cohomology products. In Section~\ref{sec:proof} we prove the \emph{Resonance
Theorem}~\ref{thm:resonance}. 

The proof depends on the results of Section~\ref{sec:critical} and Section~%
\ref{sec:mean}, and upon three Lemmas~\ref{lem:A}, ~\ref{lem:B}, ~\ref{lem:C}
regarding the loop homology and cohomology of $\Lambda S^{n}$. In Section~%
\ref{sec:proof_Lemma_C} we explain how Lemma~\ref{lem:A} and Lemma~\ref%
{lem:B} follow from the results of~\cite{GH} and prove Lemma~\ref{lem:C}. We
also compute the Chas-Sullivan product on $H_{\ast }\left( \Lambda (S^{n});%
\mathbb{Z}_{2}\right) $ for $n>2$ even.

In Section~\ref{sec:continuity} we prove the Continuity and Density
Theorems, Corollary~\ref{cor:pinch} and Corollary~\ref{cor:Katok}. 
Section~\ref{sec:perturbation} contains the proof of the Open Mapping
Theorem~\ref{thm:open_mapping_theorem}. The main step is the Perturbation
Theorem~\ref{theorem:perturbation}; Lemma~\ref{lem:drei} on +-curves in the
symplectic group may be of independent interest. In Section~\ref%
{sec:ellipsoid} we prove the Ellipsoid Theorem. Appendix A (Section~\ref{sec:more}) contains proofs of several results on closed geodesics that should be familiar to experts. 
In Appendix B (Section~\ref{sec:appendixB}) we show 
that the string brackets on
the equivariant homology and cohomology of spheres are trivial. 
\section{Discussion of the Resonance Theorem~\ref{thm:resonance}}
\label{sec:discussion}

\smallskip

{\sc Average index and mean frequency of a closed geodesic:}

\smallskip
Fix a metric $g$ on $M$, with $M$ simply connected, and let $G$ be
a field. According to a theorem of Gromov~\cite[thm.7.3]{Gr}, \cite[thm.5.10]{Pa},
the points $\left(\cri(X),\deg (X)\right)$, $X\in H_{\ast
}(\Lambda M)$, lie between two lines: \ There are positive constants $C_{1}$
and $C_{2}$ such that, for any $X\in H_{\ast }(\Lambda M)$, 
\begin{equation*}
C_{1}\,\cri(X)<\deg (X)<C_{2}\,\cri(X)
\end{equation*}%
The Resonance Theorem~\ref{thm:resonance} 
states if $M$ is a sphere, there is a positive real number $%
\overline{\alpha }=\overline{\alpha }_{g,G}$ and a constant $h$ so that 
\begin{equation*}
\overline{\alpha }\,\cri(X)-h<\deg (X)<\overline{\alpha }\,\cri(X)+h.
\end{equation*}

Let $M$ be a compact Riemannian or Finsler manifold of dimension $n$. If $%
\gamma \in \Lambda M$, we denote by $\gamma ^{m}$ the $m^{th}$ iterate of $%
\gamma :\gamma ^{m}(t)=\gamma (mt)$. A closed geodesic that is not an $%
m^{th} $ iterate for any $m>1$ is called \textit{prime.} Closed geodesics
are considered \textit{geometrically distinct }if they have different images
in $M$. (For a non-reversible Finsler metric we also consider closed
geodesics distinct if they have opposite orientations.) If $\gamma $\ is a
closed geodesic on $M$ with length $L>0$, then $\gamma ^{m}$ is a closed
geodesic of length $mL.$ Bott \cite{Bo} proved that the \textit{average index%
} 
\begin{equation*}
\alpha _{\gamma }=\lim_{m\rightarrow \infty }\frac{\mathrm{ind}(\gamma ^{m})%
}{m}
\end{equation*}%
exists. Therefore so does the \textit{mean frequency} $\overline{\alpha }%
_{\gamma }$:%
\begin{equation*}
\overline{\alpha }_{\gamma }=\underset{m\rightarrow \infty }{\lim }\frac{%
\mathrm{ind}(\gamma ^{m})}{l(\gamma ^{m})}=\frac{\alpha _{\gamma }}{L}\,.
\end{equation*}%
The mean frequency can be thought of as the number of conjugate points per
unit length. It can be estimated using the sectional or flag curvature,
cf. Lemma~\ref{lem:estimate}.  If 
\begin{equation*}
\delta^2\le K\le\Delta^2
\end{equation*}%
then for every closed geodesic $\gamma $ on $M$, 
\begin{equation*}
\frac{\delta (n-1)}{\pi }\le\overline{\alpha }_{\gamma }
\le\frac{\Delta(n-1)}{\pi }.
\end{equation*}
Note also that $\overline{\alpha }_{\gamma }$ does not depend on the field,
and that \ \ 
\begin{equation*}
\overline{\alpha }_{\gamma }=\overline{\alpha }_{\gamma ^{m}}
\end{equation*}%
for $m>1$.
\\
\smallskip\\
{\sc Statement of the Resonance Theorem~\ref{thm:resonance}
 in terms of the spectral sequence:}

\smallskip
Fix a metric $g$ on $M$, and let $G$ be a field. Let $0=\ell _{0}<\ell
_{1}<\ldots$ be a sequence of real numbers each of which is a regular value, or
a nondegenerate critical value, of $F$. \ If the critical values are
isolated it will make sense to assume that there is a unique critical value
in each $(\ell _{i},\ell _{i+1}]$. \ The filtration $\{\Lambda ^{\leq \ell
_{i}}\}$ induces a spectral sequence converging to $H_{\ast }(\Lambda ).$ \
\ Each page of the homology spectral sequence is bigraded by the index set $%
\{i\}$ of the sequence $\{\ell _{i}\}$, and the whole numbers $d$ . The
units of $\ell _{i}$ are \textit{length} (the same units as critical
values), and the units of $d$ are \textit{degree} (the same units as index)
. It is convenient to think of this as a first quadrant spectral sequence in
the $(\ell ,d)$-plane indexed by $\{(\ell _{i},d)\}$. $\ $Each page of the
spectral sequence is the direct sum of its $(\ell _{i},d)$ terms. \ The $%
(\ell _{i},d)$ term of the $\mathcal{E}^{1}$ page is given by 
\begin{equation*}
H_{d}(\Lambda ^{\leq \ell _{i}},\Lambda ^{\leq \ell _{i-1}};G).
\end{equation*}%
The $k^{th}$ page is obtained from the previous page by "cancelling" terms
by the differential $D_{k}$. \ The differential $D_{k}$, $k\geq 1$ has
degree $(-k,-1)$ (so that, roughly speaking, 
$D_{k}:H_{d}(\Lambda ^{\leq\ell _{i}},
\Lambda ^{\leq \ell _{i-1}};G)
\rightarrow 
H_{d-1}(\Lambda ^{\leq \ell _{i-k}},
\Lambda ^{\leq \ell _{i-(k+1)}};G)$). 
The homology version of
the Resonance Theorem~\ref{thm:resonance} 
states that if $M$ is a sphere, in the $\mathcal{E}%
^{\infty }$ page of the spectral sequence all nontrivial entries lie within
a bounded distance of a line 
\begin{equation*}
d=\overline{\alpha }\ell .
\end{equation*}

There is also a spectral sequence converging to $H^{\ast}(\Lambda)$ , and a
similar statement for cohomology:\ The $(\ell_{i},d)$ term of the $\mathcal{E%
}^{1}$ page is given by 
\begin{equation*}
H^{d}(\Lambda^{\leq\ell_{i}},\Lambda^{\leq\ell_{i-1}};G).
\end{equation*}
and the differential $D_{k}$, $k\geq1$ has degree $(k,1)$. \ If $M$ is a
sphere, in the $\mathcal{E}^{\infty}$ page of the spectral sequence all
nontrivial entries lie within a bounded distance of the line 
\begin{equation*}
d=\overline{\alpha}\ell.
\end{equation*}

\bigskip

If all closed geodesics are isolated, and if each $(\ell _{i-1},\ell _{i}]$
contains exactly one critical value $L_{i}$, the $\mathcal{E}^{1}$ page of
the spectral sequence is naturally a direct sum of the "contributions" of
all the closed geodesics \ $\gamma $. \ \ A closed geodesic $\gamma $ of
length $L_{i}$ contributes its \textit{local homology} 
\begin{equation}
\label{eq:local_homology}
H_{\ast }(\Lambda ^{<L_{i}}\cup \group\cdot \gamma ,\Lambda ^{<L_{i}})
\end{equation}
\ to $H_{\ast }(\Lambda ^{\leq \ell _{i}},\Lambda ^{\leq \ell _{i-1}})$,
where $\group\cdot \gamma $ is the orbit of $\gamma $ under the group $%
\group$, with $\group = \mathbb{O}(2)$ for a Riemannian metric and $%
\group = S\mathbb{O}(2)=S^1$ for an non-reversible Finsler metric. \ \ If $%
X\in H_{\ast }(\Lambda ^{<L_{i}}\cup \group\cdot \gamma ,\Lambda
^{<L_{i}})$ is nontrivial, then we conclude from Lemma~\ref%
{lem:resonant_iterates} on the \emph{resonant iterates:} 
\begin{equation*}
\left|\deg X-\overline{\alpha }_{\gamma }\cri (X) \right|\leq n\,.
\end{equation*}%
Thus the contributions to the $\mathcal{E}^{1}$ page of the homology
spectral sequence (or, by the same argument, the cohomology spectral
sequence) from the iterates \ of a single prime closed geodesic $\gamma $ of
mean frequency $\overline{\alpha }_{\gamma }$ lie at most a vertical
distance $n$ from the line%
\begin{equation*}
d=\overline{\alpha }_{\gamma }\ell .
\end{equation*}

\smallskip

{\sc Examples:}

\smallskip

The resonance theorem is obvious in the following three cases:

\begin{enumerate}
\item A hypothetical metric with only one closed geodesic. In this case it
follows from the above that in the $\mathcal{E}^{1}$ page of the homology
spectral sequence, all nontrivial terms lie at a bounded distance from the
line $d=\overline{\alpha}\ell.$ Because (over a field) the $\mathcal{E}%
^{\infty}$ page is obtained by "cancelling" terms $x$ and $y$ if $Dx=y$, on
the $\mathcal{E}^{\infty}$ page all nontrivial terms again lie at a bounded
distance from the same line.

\item The round metric of constant curvature $K$. \ In this case all
geodesics are closed with (prime) length $2\pi/\sqrt{K}$ . \ The index of
the $m^{th}$ iterate of every closed goedesic is $(2m-1)(n-1)$ and therefore
as in case (1), in the $\mathcal{E}^{1}$ page of the homology spectral
sequence all nontrivial terms lie at a bounded distance from the line 
\begin{equation*}
d=\frac{\sqrt{K}(n-1)}{\pi}\ell \,,
\end{equation*}
and this property persists to $\mathcal{E}^{\infty}.$

\item The Katok metrics on $S^n,$ compare Corollary~ \ref{cor:Katok}. In
these examples the closed geodesics have different lengths and different
average indices but they all have the same mean frequency since the flag
curvatures are constant! Thus in this case also the theorem is clear from
the $\mathcal{E}^{1}$ page.
\end{enumerate}

In all other cases we find the Resonance Theorem surprising.

Consider for example the $n$-dimensional ellipsoid%
\begin{equation*}
M=M\left( a_{0},a_{1},\ldots a_{n}\right) =\left\{ (x_{0},x_{1},\ldots
,x_{n})\in \mathbb{R}^{n+1}\,\,:\,\,\frac{x_{0}^{2}}{a_{0}^{2}}+\frac{%
x_{1}^{2}}{a_{1}^{2}}+...+\frac{x_{n}^{2}}{a_{n}^{2}}=1\right\}
\end{equation*}%
where $a_{0}<a_{1}<...<a_{n}$.
The ellipsoid carries 
$\binom{n+1}{2}=(n^{2}+n)/2$ \emph{short} closed geodesics, namely the
intersection of the ellipsoid with any of the coordinate planes. However
Morse proved that in the limit as all $a_{i}\rightarrow 1$, the length of
the next shortest prime closed geodesic goes to infinity, cf.~\cite[Lem.3.4.7%
]{Kl82}. The geodesic flow on the ellipsoid is integrable and there are
invariant $2$-dimensional tori in the unit tangent bundle. Periodic flow
lines are dense, i.e. the unit tangent vectors of closed geodesics form a
dense subset of the unit tangent bundle. In particular there are infinitely
many closed geodesics, cf.~\cite[Sec.3.5]{Kl82}. It would be very
interesting to know a formula for the mean frequencies of the $(n^{2}+n)/2$
short closed geodesics in terms of the $(n+1)$ parameters $a_{j}$! \ Our
intuition is that for a generic ellipsoid, the mean frequencies of the \emph{%
short} closed geodesics are all different. This would mean that the $%
\mathcal{E}^{1}$ page of the spectral sequence \ has nontrivial entries
roughly evenly spaced along $\binom{n+1}{2}$ different lines through the
origin in the $(\ell ,d)$ plane. \ But in the $\mathcal{E}^{\infty }$ page
at most one of these lines remains; all high iterates of all but at most one
of the $\binom{n+1}{2}$ short geodesics must be \emph{killed off} before the 
$\mathcal{E}^{\infty }$ page. This would mean that the geodesics that
dominate at low energy are irrelevant at high energy. While we have not had
much success computing the mean frequencies on an ellipsoid, the
{\em Ellipsoid Theorem~\ref{thm:ellipsoid}}
shows that for an open set of metrics, many of the mean
frequencies of the short closed geodesics on an ellipsoid are distinct.

\section{Critical levels, duality, and coefficients}

\label{sec:critical} In this section let $X$ \ be a Hilbert manifold
and let $F:X\rightarrow \mathbb{R}$ be a smooth function satisfying
condition C. For $a\in \mathbb{R}$ let $X^{a}=\{x\in X;F(x)\leq a\}$, and $%
X^{a-}=\{x\in X;F(x)<a\}$. \ The critical levels of a relative homology or
cohomology class$\ X\in H_{m}(X^{c},X^{a};G)$ or $x\in H^{m}\left(
X^{c},X^{a};G\right) $ are defined by%
\begin{eqnarray*}
\cri(X) &=&\cri_{G}(X)=\inf \{b\in
\lbrack a,c]:x\text{ is in the image of }H_{m}\left( X^{b},X^{a};G\right) \};
\\
\cri(x) &=&\cri_{G}(x)=\sup \{b\in
\lbrack a,c]:x\text{ is in the image of }H^{m}\left( X^{c},X^{b-};G\right) \}
\\
&=&\sup \{b\in \lbrack a,c]:x\rightarrow 0\text{ in }H^{m}\left(
X^{b-},X^{a};G\right) \}.
\end{eqnarray*}%
The critical level of a homology or cohomology class is a critical value of $%
F$.


\begin{lemma}
\label{lem:duality} \textrm{\textsc{(Duality Lemma) }} Let $G$ be a field.
Suppose $H_{m}\left( X^{c},X^{a};G\right) =G,$ and that $\ X\in
H_{m}(X^{c},X^{a};G)$ and $x\in H^{m}\left( X^{c},X^{a};G\right) $ have
nontrivial$\,$Kronecker product $x(X)$. \ Then $\cri(X)=\cri(x)\,.$
\end{lemma}


\begin{proof}
\ Clearly 
\begin{equation}
\cri(x)\leq \cri(X).  \label{eq:YYY}
\end{equation}%
Suppose $a<b<c$, and let $j:(X^{b},X^{a})\rightarrow (X^{c},X^{a})$. \ If 
$\cri(x)<b$, then $j^{\ast }(x)\in H^{m}(X^{b},X^{a};G)$ is
nontrivial. \ It follows that there is a class $Y\in H_{m}(X^{b},X^{a};G)$
with $j^{\ast }(x)(Y)\neq 0$. But then $x(j_{\ast }(Y))=j^{\ast }(x)(Y)\neq
0,$ from which it follows that $j_{\ast }(Y)$ is a multiple of $X$, and thus 
$\cri(X)\leq b.$
\end{proof}


\bigskip

\begin{remark} \rm
Things could be more complicated if the groups have rank $>1.$
Assume that $H_{m}(X^{c},X^{a};G)=G^{k}.$ For simplicity assume that $a$ and 
$c$ are not critical values. There will be critical values $b_{1},\ldots,b_{j}$
with $a<b_{1}<...<b_{j}<c$ and subspaces 
\begin{equation*}
A_{1}\subset A_{2}\subset ...\subset A_{j}=H_{m}(X;G)
\end{equation*}%
where $A_{i}$ is the image of $H_{m}(X^{b_{i}};G)$ in $H_{m}(X;G).$, and
where $A_{i}/A_{i-1}$ has dimension $k_{i}$, with $\Sigma _{i=1}^{j}k_{i}=k$%
, and so that if $Y\in (A_{i}-A_{i-1})$, then $\cri (Y)=b_{i}$.

Similarly there are subspaces 
\begin{equation*}
B_{j}\subset B_{j-1}\subset...\subset B_{1}=H^{m}(X;G)
\end{equation*}
where $B_{i}$ is the kernel of the map $H^{m}(X;G)%
\rightarrow H^{m}(X^{<b_{i}};G),$ so that $B_{i}/B_{i+1}$ has dimension $%
k_{i}$, and $B_{i}-B_{i+1}$ is the set of cohomology classes with critical
value $b_{i}$. 

Note also that with respect to the Kronecker product, 
\begin{equation*}
B_{i}(A_{i-1})=0,
\end{equation*}
and this equation could be used to define $B_{i}$ once the $A_{i}$ are
defined. For a generic basis for $H_{m}(X^{c},X^{a};G)$, the critical value
of each generator would be high ($=b_{k}$), and for the dual basis (or a
generic basis) for $H^{m}(X^{c},X^{a};G)$, the critical value of each
generator would be low ($=b_{1}$). 
\end{remark}


\begin{lemma}
\label{lem:effect} \textrm{\textsc{(Lemma on the effect of
Coefficients)}} Suppose $H^{m}(X^{c},X^{a})=:H^{m}(X^{c},X^{a};\mathbb{Z})=%
\mathbb{Z}$ and $H_{m}(X^{c},X^{a})=:$ $H_{m}(X^{c},X^{a};\mathbb{Z})=%
\mathbb{Z}\,.$ Let $z,Z$ be generators of these groups, and let $w\in
H^{m}(X^{c},X^{a};G)$ and $W\in H_{m}(X^{c},X^{a};G)$ be nontrivial, with 
$G=\mathbb{Z}$ or a field. Then 
\begin{align*}
\cri_{\mathbb{Z}}(z)& \leq \cri%
_{G}(w)\leq \cri_{\mathbb{Z}}(Z) \\
\cri_{\mathbb{Z}}(z)& \leq \cri%
_{G}(W)\leq \cri_{\mathbb{Z}}(Z)
\end{align*}
\end{lemma}


\begin{proof}
First $H^{m}(X^{c},X^{a})=\mathbb{Z}$ 
implies~\cite[Cor.7.3]{Br} that $H_{m-1}(X^{c},X^{a})$ is torsion
free. From this it follows~\cite[p.278]{Br} that 
$H_{m-1}(X^{c},X^{a})\ast G=0$, and thus~\cite
[7.5]{Br} that 
\begin{equation*}
H_{m}(X^{c},X^{a})\otimes G\overset{\approx }{\rightarrow }%
H_{m}(X^{c},X^{a};G)=G.
\end{equation*}
So $W=Z\otimes g$ for some (nontrivial) $g\in G$. \ It follows that $\cri_{G}(W)
\leq \cri_{\mathbb{Z}}(Z)$ 
(since for every representative $R$ for the homology class $Z$, we
have the representative $R\otimes g$ for $W$ , and $\rm{Supp}(R\otimes g)\subset
{\rm Supp}(R)$.) \ \ Moreover~\cite[p.278]{Br} 
${\rm Ext} \left(H_{m-1}\left(X^{c},X^{a}\right),G\right)=0$, so 
(~\cite[7.2]{Br}). 
\begin{equation*}
H^{m}(X^{c},X^{a};G)\overset{\approx }{\rightarrow }%
{\rm Hom}(H_{m}(X^{c},X^{a}),G)=G,
\end{equation*}
and when $G=\mathbb{Z}$ 
the latter isomorphism $\overset{\approx }{\rightarrow }$ is induced by
the Kronecker product. \ Thus $z(Z)=1$. \ \ This isomorphism is also natural
in the coefficients, so the map $\mathbb{Z}\rightarrow G$ 
taking $1$ to the identity of $G$ induces the map 
$H^{m}(X^{c},X^{a})\rightarrow H^{m}(X^{c},X^{a};G)$ 
taking $1$ to the
identity of $G$. \ Thus $w=hz$ (as maps $H_{m}(X^{c},X^{a})\rightarrow G$)
for some (nontrivial) $h\in G$. It follows that 
$\cri$ $_{\mathbb{Z}}(z)\leq \cri_{G}(w)$. 
If $G=\mathbb{Z}$, 
we use the fact that the Kronecker products $w(Z)=hz(Z)$ and 
$z(W)=gz(Z) $ are nontrivial to conclude that 
$\cri_{\mathbb{Z}}(w)
\leq \cri_{\mathbb{Z}}(Z)$ and 
$\cri_{\mathbb{Z}}(z)
\leq \cri_{\mathbb{Z}}(W)$. 
When $G$ is a field we have from {\em Duality Lemma}~\ref{lem:duality} that 
$\cri_{G}(w)=\cri_{G}(W)$, 
and the Lemma follows.
\end{proof}


\begin{remark}
\rm
It is not difficult to imagine a scenario in which the critical
level of a homology class depends on the coefficients; for example it could
be that a homology class $X$ \ has critical level $a$, but that some
multiple of $X$ is homologous to a class at a lower level. We do not know an
example of a loop space where this happens. \ The simplest case is an
interesting question of basic geometry: Let $X\in H_{n-1}(\Lambda S^{n})$ be
a generator. Is there a metric $g$ on $S^{n}$ and an integer $k$ so that 
\begin{equation*}
\cri_{g,\mathbb{Z}}(kX)<\cri_{g,\mathbb{Z}}(X)\,?
\end{equation*}
\end{remark}

\section{Mean critical levels for homology and cohomology}

\label{sec:mean}

Fix a Riemannian or Finsler metric $g$ on a compact Riemannian manifold $M$
of dimension $n$. \ Fix a coefficient ring $G$, either a field or $\mathbb{Z}$. 
Let $\Lambda $ be the free loopspace and $\Lambda ^{0}\subset \Lambda $
the constant loops. \ The loop products $\bullet $ and $\circledast $ of
Chas-Sullivan~\cite{CS} and Sullivan~\cite{Su} 
\begin{alignat}{5}
\bullet & :&H_{j}(\Lambda )\otimes H_{k}(\Lambda )&\longrightarrow&
H_{j+k-n}&(\Lambda )\nonumber \\
\circledast & : &\enspace H^{j}(\Lambda ,\Lambda ^{0})
\otimes H^{k}(\Lambda ,\Lambda
^{0})&\longrightarrow&\enspace H^{j+k+n-1}&(\Lambda ,\Lambda ^{0})
\label{eq:products}
\end{alignat}
are commutative up to sign. 
It is shown in~\cite[Prop.5.3, Cor. 10.1]{GH} that they satisfy the
following basic inequalities:%
\begin{eqnarray}
\cri(X\bullet Y) &\leq &\cri(X)+
\cri(Y)\text{ for all }X,Y\in H_{\ast }(\Lambda )  \notag
\label{eq:XY} \\
\cri(x\circledast y) &\geq &\cri(x)+
\cri (y)\,\,\,\text{ for all }x,y\in H^{\ast }(\Lambda
,\Lambda ^{0}),
\end{eqnarray}%
The proofs given there are also valid in the Finsler case. The sequences 
$\cri(\tau ^{\circledast m})$ 
and $\cri%
(x\circledast \tau ^{\circledast m})$ for \textit{cohomology} classes $\tau
,x$ are \textit{always} increasing (if finite) by Equation~\ref{eq:XY}. \ As
far as we know the sequences $\cri(U^{\bullet m})$ \ and 
$\cri(X\bullet U^{\bullet m})$\ for $U,X\in H_{\ast
}(\Lambda )$ are not necessarily increasing.

We consider the limit as $m\rightarrow \infty $ of 
\begin{equation*}
\frac{\cri(\eta ^{\circledast m})}{m}\text{ and }
\frac{\cri\left(z\circledast \eta ^{\circledast m}\right)}{m}
\end{equation*}
where $\eta $ \ and $z$ are both homology or cohomology classes with $\eta $
nonnilpotent, and where $\ast $ is the appropriate loop product. Because $%
\deg (\eta ^{\circledast m})$ and
$\deg (z\circledast \eta ^{\circledast m})$ are approximately
linear, the above limits exist if and only if the limits 
\begin{equation*}
\frac{\cri(\eta ^{\circledast m})}{\deg (\eta ^{\circledast m})}\text{
and }
\frac{\cri\left(z\circledast \eta ^{\circledast m}\right)}
{\deg (z\circledast \eta^{\circledast m})}
\end{equation*}%
exist. (If $G$ is a field, these limits are determined by the Resonance
Theorem.)

%

\ Let $U,Z\in H_{\ast }(\Lambda )$ with $\deg U>n\,.$ We define the \emph{%
mean levels } 
\begin{align*}
\overline{\mu _{U}}(Z)& =\overline{\mu _{g,G,U}}(Z)=\limsup_{m\rightarrow
\infty }\,\cri_{g,G}(U^{\bullet m}\bullet Z)/m \\
\underline{\mu _{U}}(Z)& =\underline{\mu _{g,G,U}}(Z)=\liminf_{m\rightarrow
\infty }\,\cri_{g,G}(U^{\bullet m}\bullet Z)/m
\end{align*}%
%
%
%
%

When the two limits coincide we denote the common limit $\mu _{U}(Z)$.


\begin{lemma}
\label{lem:meanlevel} 
If $j\geq 0,$ and $U,X,Y\in H_{\ast }(\Lambda ),$ then 
\begin{eqnarray}
\overline{\mu _{U}}(X) &=&\overline{\mu _{U}}(U^{\bullet j}\bullet X)
\label{eq:AAA} \\
\overline{\mu _{U}}(X\bullet Y) &\leq &\min \{\overline{\mu _{U}}(X),%
\overline{\mu _{U}}(Y)\} \label{eq:CD}\\
\overline{\mu _{U}}(X) &\leq &\overline{\mu _{U}}(U) \label{eq:AB}
\end{eqnarray}%
Similar inequalities also hold for $\underline{\mu _{U}}$ and $\mu _{U}.$
\end{lemma}


\begin{proof}
Equation~\ref{eq:AAA} follows easily from the definitions. 
Equation~\ref{eq:AB} follows from the first two equations. 
We obtain Equation~\ref{eq:CD}
as follows: 
\begin{align*}
\overline{\mu _{U}}(X\bullet Y)& =
\limsup_{m\to \infty } \cri\left(U^{\bullet m}\bullet X\bullet Y\right)/m \\
& \leq \limsup_{m\rightarrow \infty }\left\{ \cri
(U^{\bullet m}\bullet X)/m+\cri(Y)/m\right\} \\
& =\overline{\mu _{U}}(X).
\end{align*}
\end{proof}

We also have mean levels for cohomology; 
\begin{equation*}
\overline{\mu _{\tau }}(z)=\limsup_{m\rightarrow \infty }\,
\cri\left(\tau ^{\circledast m}\circledast z\right)/m
\enspace;\enspace\underline{\mu
_{\tau }}(z)=\liminf_{m\rightarrow \infty }\cri(\tau
^{\circledast m}\circledast z)/m
\end{equation*}%
If the limits coincide, we denote the common limit $\mu _{\tau }(z)$. \ If $%
j\geq 0$, and $\tau ,x,y\in H^{\ast }(\Lambda ,\Lambda ^{0})$, then
similarly, 
\begin{align}
\overline{\mu _{\tau }}(x)& =\overline{\mu _{\tau }}(\tau ^{\circledast
j}\circledast x)  \label{eq:EF} \\
\overline{\mu _{\tau }}(x\circledast y)& \geq \min \{\overline{\mu _{\tau }}%
(x),\overline{\mu _{\tau }}(y)\}  \notag \\
\overline{\mu _{\tau }}(x)& \geq \overline{\mu _{\tau }}(\tau )\,.  \notag
\end{align}%
Similar inequalities also hold for $\underline{\mu _{\tau }}$ and $\mu
_{\tau }$. 

\begin{lemma}
\label{lem:powers} \textrm{\textsc{(Powers Lemma) }} \, If $U\in
H_{\ast}(\Lambda)\,,$ then $\overline{\mu_{U}}(U)=\underline{\mu_{U}}(U)$
and $\mu_{U}(U)$ exists. \ Moreover, 
\begin{equation*}
\mu_{U}(U)\leq \cri(U^{\bullet m})/m
\end{equation*}
for all $m\geq 1\,.$

If $\tau\in H^{\ast}(\Lambda,\Lambda^{0})\,,$ then $\overline{\mu_{\tau}}%
(\omega)=\underline{\mu_{\tau}}(\tau)$ and $\mu_{\omega}(\omega)$ exists. \
Moreover, 
\begin{equation*}
\mu_{\tau}(\tau)\geq \cri(\tau^{\circledast m})/m
\end{equation*}
for all $m\geq 1\,.$
\end{lemma}


\begin{proof}
The lemma follows from Fekete's Lemma~\cite{PS} 
and the fact that the sequences and 
$\cri(U^{\bullet m})$ and $
\cri(\tau
^{\circledast m})$ are subadditive and superadditive. 
\end{proof}

\textsc{Index interval}

Fix a metric on $M$ and fix $G$, either a field or $\mathbb{Z}$. 
The same arguments used in~\cite[Thm.6.2]{Rad95} show the following: 

\begin{proposition}
\label{pro:muX} Let $U,X\in H_{\ast}(\Lambda)$, with $\deg U>n$. \ If $%
t\in\lbrack\underline{\mu_{U}}(X),\overline{\mu_{U}}(X)]$ is nonzero (which
implies that $U$ is nonnilpotent), then there is a sequence $\gamma_{m}$ of
closed geodesics on $M$ whose mean frequencies $\overline {\alpha}_{m}$
converge to $(\deg U-n)/t$.

Let $\tau ,x\in H^{\ast }(\Lambda ,\Lambda ^{0}).$ \ If $t\in \lbrack 
\underline{\mu _{\tau }}(x),\overline{\mu _{\tau }}(x)]$ is finite (which
implies that $\tau $ is nonnilpotent), then there is a sequence $\gamma _{i}$
of closed geodsics on $M$ whose mean frequencies $\overline{\alpha }_{i}$
converge to $(\deg \tau +1-n)/t$.
\end{proposition}


\begin{corollary}
\label{cor:muX} If $[\underline{\mu _{U}}(X),\overline{\mu _{U}}(X)]$ or $[%
\underline{\mu _{\tau }}(x),\overline{\mu _{\tau }}(x)]$ is not a point,
then $M$ has infinitely many closed geodesics. \ 
\end{corollary}


\begin{proof}
(\emph{of the Proposition:}) \newline
We have 
\begin{equation*}
\deg (X\bullet U^{\bullet m})=\deg X+m(\deg U-n).
\end{equation*}%
\ Thus 
\begin{align*}
\limsup_{m\rightarrow \infty }\cri(X\bullet U^{\bullet
m})/\deg (X\bullet U^{\bullet m})& =\overline{\mu _{U}}(X)/(\deg U-n) \\
\liminf_{m\rightarrow \infty }\cri(X\bullet U^{\bullet
m})/\deg (X\bullet U^{\bullet m})& =\underline{\mu _{U}}(X)/(\deg U-n)
\end{align*}%
The basic inequalities~\ref{eq:XY} imply that 
\begin{align*}
\frac{\cri(X\bullet U^{\bullet (m+1)})}{\deg (X\bullet
U^{\bullet (m+1)})}-\frac{\cri(X\bullet U^{\bullet m})}{\deg
(X\bullet U^{\bullet m})}& \leq \frac{\cri(U)}{\deg
(X\bullet U^{\bullet m})} \\
& \leq \frac{1}{m}\frac{\cri(U)}{(\deg U-n)}.
\end{align*}%
The sequence 
\begin{equation*}
a_{m}=\frac{\cri(X\bullet U^{\bullet m})}{\deg (X\bullet
U^{\bullet m})}
\end{equation*}%
has the property that $a_{m+1}-a_{m}\leq \varepsilon _{m}$, with $%
\varepsilon _{m}\rightarrow 0$. \ It follows that the set of limit points of
this sequence is an interval; thus every point in 
\begin{equation*}
\left[ \frac{\underline{\mu _{U}}(X)}{\deg U-n}\,,\,\frac{\overline{\mu _{U}}(X)%
}{\deg U-n}\right]
\end{equation*}%
is a limit point of the sequence.  For each $m$ there is a closed geodesic 
$\gamma _{m}$ of length $\cri(X\bullet U^{\bullet m})$, whose index $\lambda
_{m}$ satisfies%
\begin{equation*}
\deg (X\bullet U^{\bullet m})-2n+1\leq \lambda _{m}\leq \deg (X\bullet
U^{\bullet m}).
\end{equation*}%
Thus (inverting, and using the fact that $l(\gamma _{m})$ and $\deg
(X\bullet U^{\bullet m})$ $\rightarrow \infty ),$ the sequence $\{\lambda
_{m}/\ell(\gamma _{m})\}$ has the same limit points as $\{\deg (X\bullet
U^{\bullet m})/cr(X\bullet U^{\bullet m})\},$ i.e. the interval 
\begin{equation*}
\left[ \frac{\deg U-n}{\overline{\mu _{U}}(X)}\,,
\,\frac{\deg U-n}{\underline{%
\mu _{U}}(X)}\right] \,.
\end{equation*}%
Using the \emph{Resonant Iterates Lemma}~\ref{lem:resonant_iterates} we find 
\begin{equation*}
\left\vert \frac{\mathrm{ind}\gamma ^{m}}{l(\gamma _{m})}-\overline{\alpha }%
_{m}\right\vert \leq \frac{n}{l(\gamma _{m})}
\end{equation*}%
and thus the sequence $\{\overline{\alpha }_{m}\}$ has the same interval of
limit points, and the lemma is proved for homology. \ 

The argument for cohomology is similar; in particular $\deg (x\circledast
\tau ^{\circledast m})=\deg x+m(\deg \tau +n-1)$ and, using the \emph{Powers
Lemma~\ref{lem:powers}}: 
\begin{align*}
\frac{\cri(x\circledast \tau ^{\circledast (m+1)})}{\deg
(x\circledast \tau ^{\circledast (m+1)})}-\frac{\cri%
(x\circledast \tau ^{\circledast m})}{\deg (x\circledast \tau ^{\circledast
m})}& \geq -\frac{\cri(x\circledast \tau ^{\circledast m})%
}{m^{2}(\deg \tau +n-1)} \\
& \geq -\frac{\overline{\xi _{\tau }}(x)}{m(\deg \tau +n-1)}
\end{align*}%
so that if $\overline{\xi _{\tau }}(x)<\infty $, the sequence $a_{m}=$ $
\cri (x\circledast \tau ^{\circledast m})/\deg (x\circledast
\tau ^{\circledast m})$ has the property that $a_{m}-a_{m+1}\leq \varepsilon
_{m}$, with $\varepsilon _{m}\rightarrow 0$.
\end{proof}


\section{Proof of the Interval Theorem and Resonance Theorem}

\label{sec:proof}

Our first task in this section is to prove:
\begin{theorem}
\label{thm:interval} \textrm{\textsc{(Interval Theorem) }} Fix a metric on $%
S^{n},n>2,$ and let $G\,$\ be a field or $\mathbb{Z}.$ Let 
\begin{equation*}
\mu ^{-}=:\lim_{m\rightarrow \infty }\frac{\cri_{\mathbb{Z}%
}(\omega ^{\circledast m})}{m}\enspace;\enspace\mu ^{+}=\lim_{m\rightarrow
\infty }\frac{\cri_{\mathbb{Z}}(\Theta ^{\bullet m})}{m}
\end{equation*}%
where $\Theta $ and $\omega $ are the fundamental nonnilpotent elements in
loop homology and cohomology, cf. Lemma~\ref{lem:B} below. Any sequence of
homology or cohomology classes $\{X_{k}\}$ with $X_{k}\in H_{\ast }\left(
\Lambda ,\Lambda ^{0};G\right) $ or $X_{k}\in H^{\ast }\left( \Lambda
,\Lambda ^{0};G\right) ,$ with $\lim_{k\rightarrow \infty }\deg X_{k}=\infty
,$ has 
\begin{equation*}
\left[ \liminf_{k\rightarrow \infty }
\frac{\cri_{G}(X_{k})}{\deg X_{k}}\,,\,
\limsup_{k\rightarrow \infty }
\frac{\cri_{G}(X_{k})}{\deg X_{k}}\right] \,
\subseteq \,\left[ \frac{\mu ^{-}}{2(n-1)}\,,\,\frac{\mu ^{+}}{2(n-1)}\right]
\end{equation*}
\end{theorem}
Then we will use the interval theorem to prove the following theorem already
stated as Theorem~\ref{thm:resonance} in the Introduction:
\begin{theorem}
\label{thm:resonance5} 
\thinspace \textrm{\textsc{(Resonance Theorem)}} \ A
Riemannian or Finsler metric $g$ on $S^{n},n>2$, and a field $G$ determine a 
\emph{global mean frequency } $\overline{\alpha }=:\overline{\alpha }%
_{g,G}>0 $ with the property that 
\begin{equation*}
\deg (X)-\overline{\alpha }\,\cri(X)
\end{equation*}%
is bounded as $X$ ranges over all nontrivial homology and cohomology classes
on $\Lambda $ with coefficients in $G$.\newline
\end{theorem}


Fix a metric $g$ on $M=S^{n}$, $n>2$. The interval theorem will follow
from basic properties of loop products, the results of Sections 1 and 2, and
two lemmas. The lemmas,   Lemma~\ref{lem:A} and Lemma~\ref{lem:B}, are
proved in~\cite{GH}, and will be discussed in Section~\ref{sec:proof_Lemma_C}.
 For the Resonance Theorem~\ref{thm:resonance5} 
we need a third lemma; Lemma~\ref{lem:C} will
be proved in Section~\ref{sec:proof_Lemma_C}.

\medskip\ \ 

\begin{lemma}
\label{lem:A} For each $j,$ each of $H^{j}(\Lambda,\Lambda^{0};\mathbb{Z})$, 
$H^{j}(\Lambda;\mathbb{Z}),$ and $H_{j}(\Lambda;\mathbb{Z})$ is one of the
following: $0$, $\mathbb{Z}$, or $\mathbb{Z}_{2}=:\mathbb{Z}/2\mathbb{Z}$.
If $G$ is a field, 
\begin{equation*}
\mathrm{rank\,} H^{j}(\Lambda,\Lambda^{0};G)\leq \mathrm{rank\,}
H^{j}(\Lambda;G)=\mathrm{rank\,} H_{j}(\Lambda;G)\leq 1
\end{equation*}
\end{lemma}

This lemma allows us to use the \emph{Duality Lemma~\ref{lem:duality} }, and
the \emph{Lemma}~\ref{lem:effect} \emph{on the effect of coefficients.} 

\begin{lemma}
\label{lem:B} Fix $G$, a field or $%
\mathbb{Z}
$. There is a homology class $\Theta \in H_{3n-2}(\Lambda ;G)$ and a
cohomology class $\omega \in H^{n-1}(\Lambda ,\Lambda ^{0};G),$ with the
property that $\Theta ^{\bullet m}$ is a generator of $H_{(2m+1)(n-1)+1}(%
\Lambda ;G)=G$, and $\omega ^{\circledast m}$ is a generator of $%
H^{(2m-1)(n-1)}(\Lambda ,\Lambda ^{0};G)=G.$ Multiplication with $\Theta $
and multiplication with $\omega $ are dual: 
\begin{equation}
x\circledast \omega ^{\circledast m}(Y\bullet \Theta ^{\bullet m})=x(Y)
\label{eq:9999}
\end{equation}%
for every $x\in H^{\ast }(\Lambda ;G)$, $Y\in H_{\ast }(\Lambda ,\Lambda
^{0};G)$ and $m\geq 0$. There is a finite set $Q$ of classes of degree $\leq
3n-2$ in $H_{\ast }(\Lambda ;G),$ and a finite set $q$ of classes in $%
H^{\ast }(\Lambda ,\Lambda ^{0};G)$ so that every class $X\in H_{j}(\Lambda
;G)$ with degree $j>n$ can be written uniquely in the form $X=aY\bullet
\Theta ^{\bullet m}$ for some $a\in G$ and $Y\in Q,$ and every class $x\in
H^{j}(\Lambda ,\Lambda ^{0};G)$ can be written uniquely in the form $%
x=ay\circledast \omega ^{\circledast m}$ for some $a\in G$ and $y\in q.$ If $%
G$ is a field, these can be chosen so that for every $Y\in Q$ (resp. $y\in q$%
) there is a unique $y\in q$ (resp. $Y\in Q$ ) with $y(Y)=1\,.$
\end{lemma}

Note that $\bullet \Theta $ and $\circledast \omega $ both increase degree
by $2(n-1)$. Let \ $W$ be the generator of $H_{n-1}(\Lambda ;G)$, so that $%
\omega (W)=1$ then $W\bullet \Theta ^{\bullet m}$ is a generator of $%
H_{(2m+1)(n-1)}(\Lambda ;G)=G$, and 
\begin{equation}
\omega ^{m+1}(W\bullet \Theta ^{\bullet m})=1\,.  \label{eq:QXR}
\end{equation}%
Similarly, if $\theta $ is the generator of $H^{3n-2}(\Lambda ,\Lambda ^{0})$%
, so that $\theta (\Theta )=1$ then 
\begin{equation}
\theta \circledast \omega ^{m-1}(\Theta ^{\bullet m})=1\text{\thinspace }.
\label{eq:99}
\end{equation}

Fix an integral domain $G$. \ Let 
\begin{equation*}
\mu _{G}^{+}=\lim_{m\rightarrow \infty }\frac{\cri%
_{G}(\Theta ^{\bullet m})}{m}\,\,;\,\,\mu _{G}^{-}=\lim_{m\rightarrow \infty
}\frac{\cri_{G}(\omega ^{\circledast m})}{m}
\end{equation*}%
(These limits exist by the \emph{Power Lemma~\ref{lem:powers}}.) We will
first establish that 
\begin{equation}
\mu ^{-}=:\mu _{\mathbb{Z}}^{-}\leq \mu _{G}^{-}\leq \mu _{G}^{+}\leq \mu _{%
\mathbb{Z}}^{+}\text{ }=:\mu ^{+}\,.  \label{eq:33}
\end{equation}
The first and third inequalities follow from the \emph{Lemma~\ref{lem:effect}
on effect of Coefficients}, and Lemma~\ref{lem:B}. From now on in this
section 
\begin{equation*}
\cri=:\cri_{G}\text{.}
\end{equation*}%
We have \ 
\begin{align*}
\lim_{m\rightarrow \infty }\frac{\cri(\Theta ^{\bullet m})%
}{m}& =:\mu _{\Theta }(\Theta )\text{ \ \ \ \ \ \ (by definition, Equation~%
\ref{eq:AB})} \\
& \geq \overline{\mu _{\Theta }}(W)\text{ \ \ \ \ \ \ \ ( Equation~\ref%
{eq:AAA})} \\
& =\limsup_{m\rightarrow \infty }\frac{\cri(W\bullet
\Theta ^{\bullet m})}{m}\text{ \ \ \ \ (by definition)} \\
& \geq \limsup_{m\rightarrow \infty }\frac{\cri(\omega
^{\circledast m+1})}{m}\text{ \ \ \ \ \ \ \ (Equation~\ref{eq:YYY} and 
Equation~\ref{eq:QXR})} \\
& =\limsup_{m\rightarrow \infty }\frac{\cri(\omega
^{\circledast m})}{m}\text{ \ \ \ \ \ \ \ \ \ \ \ (Equation~\ref{eq:EF})} \\
& =\lim_{m\rightarrow \infty }\frac{\cri(\omega
^{\circledast m})}{m}\text{ \ \ \ \ \ (Power Lemma~\ref{lem:powers})}
\end{align*}%
which proves Equation~\ref{eq:33}. Let $Q,q$ be as in Lemma \ref{lem:B}. \
For every $Y\in Q$ and $y\in q$, by the basic inequalities~\ref{eq:XY}, 
\begin{align}
\cri(\omega ^{\circledast m})& 
\leq \cri
(y\circledast \omega ^{\circledast m})  \label{eq:87} \\
\cri(Y\bullet \Theta ^{\bullet m})& 
\leq \cri(Y)+\cri(\Theta ^{\bullet m})  \label{eq:88}
\end{align}%
We will show that the Interval Theorem~\ref{thm:interval} follows,
considering seperately the cases when $G$ is a field and $G=\mathbb{Z}$.  

If $G$ is a field, given $\ X\in H_{j}(\Lambda ;G)$ nontrivial with $j>n$,
write $X=Y\bullet \Theta ^{\bullet m}$\ for $Y\in Q$; suppose $y(Y)=1$. By
the Duality Lemma~\ref{lem:duality}, and Equation~\ref{eq:9999}, 
\begin{equation*}
\cri_{G}(y\circledast \omega ^{\circledast m})
=\cri (X)=\cri_{G}(Y\bullet \Theta ^{\bullet m})
\end{equation*}%
and thus with Equation~\ref{eq:87} and Equation~\ref{eq:88}, we have 
\begin{equation}
\cri(\omega ^{\circledast m})\leq \cri%
(X)\leq \cri(Y)+\cri(\Theta ^{\bullet
m})\,.  \label{eq:QQ}
\end{equation}%
Since 
\begin{equation*}
\lim_{m\rightarrow \infty }\frac{\deg (y\circledast \omega ^{\circledast m})%
}{m}=\lim_{m\rightarrow \infty }\frac{\deg (Y\bullet \Theta ^{\bullet m})}{m}%
=2(n-1)
\end{equation*}%
for all $y$, $Y$, the Interval Theorem~\ref{thm:interval} is a consequence.

If $G=\mathbb{Z},$ the Interval Theorem for homology/cohomology in
dimensions in which the homology and cohomology are $\mathbb{Z}$ follows
from Lemma~\ref{lem:A}, Lemma~\ref{lem:B}, Equation~\ref{eq:87}, \ and
Equation~\ref{eq:88} by the Lemma~\ref{lem:effect} on effect of
Coefficients. Note the theorem applies since $\Theta ^{\bullet m}$ and $%
\omega ^{\circledast m}$ are generators in their dimensions. If $X$ $\in
H_{k}(\Lambda ;\mathbb{Z})=\mathbb{Z}_{2},$ note that the reduction $%
\overline{X}\in H_{k}(\Lambda ;\mathbb{Z}_{2})$ of $X$ mod $2$ satisfies 
$\cri_{\mathbb{Z}_{2}}(\overline{X})\leq \cri_{\mathbb{Z}} X$. 
We leave the remaining details of the case $G=\mathbb{Z}
$ to the reader. 

For the proof of the resonance theorem we will need also the following: \
Let 
\begin{equation*}
\Delta :H_{j}(\Lambda ;G)\rightarrow H_{j+1}(\Lambda ;G)
\end{equation*}
be the map induced by the $S^{1}-$action and the K\"{u}nneth map (\cite[p.16]%
{CS}, \cite[p.58]{GH}). \ If a homology class $X$ is supported on a closed
set $A\subset \Lambda $, then $\Delta X$ is supported on $S^{1}\cdot A$;
thus 
\begin{equation}
\mathrm{cr}(\Delta X)\leq \mathrm{cr}(X) \label{eq:Delta}
\end{equation}

\begin{lemma}
\label{lem:C} If $n$ is odd, 
\begin{equation*}
\Delta(W\bullet\Theta^{\bullet m})=(2m+1)\,\Theta^{m}\,.
\end{equation*}
If $n$ is even, then there is an integer $k=k(n)$ so that 
\begin{equation*}
\Delta(W\bullet\Theta^{\bullet m})=(mk+1)\,\Theta^{m}\,.
\end{equation*}
\end{lemma}


\begin{proof}
(of the \emph{Resonance Theorem}~\ref{thm:resonance}): Now let $G$ be a
field. Let 
\begin{equation*}
\overline{\mu }=:\mu _{G}=\lim_{m\to\infty}
 \frac{\cri_{G}(\omega
^{\circledast m})}{m}
\end{equation*}%
By Equation~\ref{eq:XY}, $\cri(\omega ^{\circledast
m})$ is increasing. By Equation~\ref{eq:XY}, Equation~\ref{eq:99} and 
the duality Lemma~\ref{lem:duality}
it follows that $\cri(\Theta
^{\bullet m})$ is also increasing. Suppose $G$ has characteristic $p$. The
coefficient $2m-1$ or $\ (m-1)k+1$ of $\Theta ^{m-1}$ in Lemma~\ref{lem:C}
cannot be congruent to $0$ mod $p$ for two consecutive values of $m$. \
Using Equation~\ref{eq:Delta}, for each $m$ at least one of the
following is true: 
\begin{equation*}
\cri\left(\Theta ^{m-1}\right)\leq \cri(\omega
^{\circledast m})\,\,;\,\,\cri\left(\Theta ^{m}\right)
\leq\cri \left(\omega ^{\circledast (m+1)}\right),
\end{equation*}
and thus for each $m$ 
\begin{equation}
\cri\left(\Theta ^{m-1}\right)\leq \cri\left(\omega
^{\circledast (m+1)}\right).  \label{eq:PP}
\end{equation}%
Let $\ X\in H_{j}(\Lambda ;G)$ be nontrivial with $j>n$; say $%
X=Y\bullet \Theta ^{\bullet m}$, with $\deg Y\leq 3n-2$ and $y(Y)\neq 0$.
By the Duality Lemma~\ref{lem:duality}, 
\begin{equation*}
\cri(X)=\cri\left(y\circledast \omega
^{\circledast m}\right).
\end{equation*}%
Using Equation~\ref{eq:PP} and Equation~\ref{eq:QQ} together with the Powers
Lemma~\ref{lem:powers} we have 
\begin{equation}
(m-2)\,\overline{\mu }\leq \cri(X)
\leq (m+2)\,\overline{\mu }+\cri(Y)\,.  \label{eq:RR}
\end{equation}
Moreover 
\begin{equation*}
\deg X=2m(n-1)+\deg Y
\end{equation*}
so 
\begin{equation}
2m(n-1)\leq \deg X\leq 2m(n-1)+3n-2 \,. \label{eq:SS}
\end{equation}%
Putting together Equation~\ref{eq:RR} and Equation~\ref{eq:SS} we have 
\begin{align*}
\frac{\overline{\mu }}{2(n-1)}\deg X\,-\,& \frac{\overline{\mu }\,(7n-6)}{2(n-1)}%
\leq\cri (X) \\
& \leq \frac{\overline{\mu }}{2(n-1)}\deg X+(\cri (Y)+2\,
\overline{\mu })+\frac{(3n-2)[\cri(Y)+2\,\overline{\mu }]}{
\deg X+2-3n}.
\end{align*}%
This is of the form 
\begin{equation*}
\frac{1}{\overline{\alpha }}\deg X-A\leq \cri(X)\leq \frac{%
1}{\overline{\alpha }}\deg X+B+\frac{C}{\deg X-D}
\end{equation*}%
with positive constants $A,B,C,D$, and 
\begin{equation*}
\overline{\alpha }=\frac{2(n-1)}{\overline{\mu }}.
\end{equation*}%
The resonance theorem follows.
\end{proof}
The following is a consequence of Theorem~\ref{thm:interval} (interval
theorem):
\begin{corollary}
\label{cor:muplus} Fix a metric $g$ on $S^{n}$ with $n>2.$ For any field $G,$
the global mean frequency $\overline{\alpha}_{g,G}$ satisfies: 
\begin{equation*}
\frac{1}{\overline{\alpha}_{g,G}} \in\left[\frac{\mu^{-}}{2(n-1)},\frac{%
\mu^{+}}{2(n-1)}\right].
\end{equation*}
\end{corollary}
It is not known whether there is a metric on $S^{n}$ with $\mu ^{-}\neq
\mu ^{+}$.\ \ 

\section{Proof of Lemma~\protect\ref{lem:C}; computation of the mod 2
Chas-Sullivan product}

\label{sec:proof_Lemma_C}

In this section we present the proof of Lemma~\ref{lem:C} and sketch the
proofs of Lemma~\ref{lem:A} and Lemma~\ref{lem:B}.

We first sum up some facts about the loop products for spheres. The
Chas-Sullivan product over the integers was computed by Cohen, Jones, and
Yan in~\cite{CJY}. We include an argument below that also works over the
field $\mathbb{Z}_{2}.$

The identity element of the homology ring is a generator $E$ of $%
H_{n}(S^{n})\subset H_{n}(\Lambda S^{n}),$ represented by the space $%
\Lambda^{0}$ of trivial loops.

For $n>1$ odd, 
\begin{equation}\label{eq:AXA1}
\left(H_{\ast}\left(\Lambda S^{n};\mathbb{Z}\right),\bullet\right)=\wedge(A)%
\otimes\mathbb{Z}\lbrack U\rbrack
\end{equation}
where $A\in H_{0}(\Lambda S^{n})$, and $U\in H_{2n-1}(\Lambda S^{n})$.

The case $n>2$ even with $\mathbb{Z}_{2}$ coefficients is formally the same: 
\begin{equation}  \label{eq:AXA}
\left(H_{\ast}\left(\Lambda S^{n};\mathbb{Z}_{2}\right),\bullet\right)=
\wedge(A)\otimes\mathbb{Z}_{2}[U]
\end{equation}
where $A\in H_{0}(\Lambda S^{n})$ and $U\in H_{2n-1}(\Lambda S^{n})$.

For the convenience of the reader, here is the $\mathcal{E}^{1}$ page of the
spectral sequence converging to the homology of $\Lambda S^{n}$ for $n$ odd,
with integer coefficients, or (respectively) for $n$ even with, mod 2
coefficients, that come from the filtration determined by the function $F$
when $M=S^{n}$ carries the round metric where the prime closed geodesics
have length $L$. \ An element in a nonempty box is a generator of homology.
\ The number at the left side of each column is the degreee $d$.

The column of a class is determined by the filtration level: For example the
term $A\bullet U^{\bullet6}$ in the right-most column represents a generator
of $H_{6n-6}(\Lambda^{\leq3L},\Lambda^{\leq2L})$ (with $\mathbb{Z}$ or,
resp., $\mathbb{Z}_{2}$ coefficients).

The column containing $A$ and $E$ is $H_{\ast}(\Lambda^{0})$. The homology
is $\mathbb{Z}$ (resp. $\mathbb{Z}_{2})$ in every spot where there is a
generator, and is $0$ otherwise. For the round metric (any $n$, any
coefficents), $F$ is a perfect Morse function, which means the spectral
sequence degenerates at the $\mathcal{E}^{1}$ page. The rank of $%
H_{j}(\Lambda S^{n})$ is thus the number of generators in the $j^{th}$ row,
and the column in which a generator appears tells the critical level of that
homology class in the round metric.

\newpage

\textsc{$H_{\ast}(\Lambda S^{n})$ ; $n$ odd or $\mathbb{Z}_{2}$ coefficients:%
} 
\begin{equation*}
\begin{array}{c|cccc}
\uparrow d &  &  &  &  \\ 
6n-5 &  &  &  & \text{ \ \ }U^{\bullet5} \\ 
&  &  &  &  \\ 
6n-6 &  &  &  & \text{ \ \ }A\bullet U^{\bullet6} \\ 
&  &  &  &  \\ 
5n-4 &  &  & \text{ \ \ }U^{\bullet4} &  \\ 
&  &  &  &  \\ 
5n-5 &  &  &  & \text{ \ \ }A\bullet U^{\bullet5} \\ 
&  &  &  &  \\ 
4n-3 &  &  & \text{ \ \ }U^{\bullet3} &  \\ 
&  &  &  &  \\ 
4n-4 &  &  & \text{ \ \ }A\bullet U^{\bullet4} &  \\ 
&  &  &  &  \\ 
3n-2 &  & \text{ \ \ }U^{\bullet2} &  &  \\ 
&  &  &  &  \\ 
3n-3 &  &  & \text{ \ \ }A\bullet U^{\bullet3} &  \\ 
&  &  &  &  \\ 
2n-1 &  & \text{ \ \ }U &  &  \\ 
&  &  &  &  \\ 
2n-2 &  & \text{ \ \ }A\bullet U^{\bullet2} &  &  \\ 
&  &  &  &  \\ 
n & \text{ \ \ \ \ \ }E\text{ \ \ \ \ \ } &  &  &  \\ 
&  &  &  &  \\ 
n-1 &  & \text{ \ \ }A\bullet U &  &  \\ 
&  &  &  &  \\ 
0 & \text{ \ \ \ \ \ \ }A\text{ \ \ \ \ \ \ } &  &  &  \\ 
&  &  &  &  \\ \hline
&  &  &  &  \\ 
& 0 & L & 2L & \text{ \ \ \ \ \ }3L\text{ }\rightarrow\ell%
\end{array}%
\end{equation*}
The relationship between this diagram and the notation of 
Section~\ref{sec:proof} is (when $n$ is odd, or $\ \textrm{\rm mod }2$)(see~\cite{GH}): 
\begin{equation*}
\Theta =U^{\bullet2}\,\,;\,\, W =A\bullet U.
\end{equation*}
The first equation may be taken for our purposes as a definition. The second
equation follows from the fact that both homology classes are generators in
their degree, and thus unique up to a sign. \ The sign can be computed
geometrically using the definitions; in any case the proof of the theorem is
not sensitive to a change in the sign of $W$.

For $n>2$ even with integer coefficients, 
\begin{equation*}
H_{\ast}(\Lambda S^{n};\bullet)= \wedge(W)\otimes\mathbb{Z}\lbrack A,
\Theta]/A^{2},AW,2A\Theta
\end{equation*}
where $W\in H_{n-1}(\Lambda S^{n}),$ $\Theta\in H_{3n-2}(\Lambda S^{n})$ as
in Section~\ref{sec:proof} and $A\in H_{0}(\Lambda S^{n})\,.$ Here is the
corresponding diagram of the spectral sequence converging to the homology of
the free loop space of an even sphere, with integer coefficients. The
classes marked $(\ast)$ are 2-torsion. \ The homology in dimensions with
generator $A$, $E$ or $W\Theta^{m}$ , or $\Theta^{m}$, $\omega^{m}$ or $%
d\omega^{m}$ is $\mathbb{Z}.$

\medskip

\textsc{$H_{\ast}(\Lambda S^{n})$ ; $n$ even and $\mathbb{Z}$ coefficients:} 
\begin{equation*}
\begin{array}{c|cccc}
\uparrow d &  &  &  &  \\ 
6n-5 &  &  &  & \text{ \ \ }0 \\ 
&  &  &  &  \\ 
6n-6 &  &  &  & \text{ \ \ }A\bullet\Theta^{\bullet3}(\ast) \\ 
&  &  &  &  \\ 
5n-4 &  &  & \text{ \ \ }\Theta^{\bullet2} &  \\ 
&  &  &  &  \\ 
5n-5 &  &  &  & \text{ \ \ }W\bullet\Theta^{\bullet2} \\ 
&  &  &  &  \\ 
4n-3 &  &  & \text{ \ \ }0 &  \\ 
&  &  &  &  \\ 
4n-4 &  &  & \text{ \ \ }A\bullet\Theta^{\bullet2}(\ast) &  \\ 
&  &  &  &  \\ 
3n-2 &  & \text{ \ \ }\Theta &  &  \\ 
&  &  &  &  \\ 
3n-3 &  &  & \text{ \ \ }W\bullet\Theta &  \\ 
&  &  &  &  \\ 
2n-1 &  & 0 &  &  \\ 
&  &  &  &  \\ 
2n-2 &  & \text{ \ \ }A\bullet\Theta(\ast) &  &  \\ 
&  &  &  &  \\ 
n & \text{ \ \ \ \ \ }E\text{ \ \ \ \ \ } &  &  &  \\ 
&  &  &  &  \\ 
n-1 &  & W &  &  \\ 
&  &  &  &  \\ 
0 & \text{ \ \ \ \ \ \ }A\text{ \ \ \ \ \ \ } &  &  &  \\ 
&  &  &  &  \\ \hline
&  &  &  &  \\ 
& 0 & L & 2L & \text{ \ \ \ \ \ }3L\text{ }\rightarrow\ell%
\end{array}%
\end{equation*}
\ 

\begin{remark}
\rm
The truth of Lemma~\ref{lem:A} and Lemma~\ref{lem:B} is evident from
the above, and the (similar) $\mathcal{E}_{1}$ page for the cohomology
spectral sequence; see~\cite{GH}. 
\end{remark}

Lemma~\ref{lem:C} will clearly follow from the following 

\begin{lemma}
\label{lem:Cast} If $n$\textit{\ is odd, \ or with }$\textrm{\rm mod }2$\textit{\
coefficients, \ } 
\begin{equation}  \label{eq:JKL}
\Delta(A\bullet U^{\bullet m})=mU^{\bullet m-1}
\end{equation}
If $n$ is even, there is an integer $k=k(n)$ so that 
\begin{equation}  \label{eq:MNO}
\Delta(W\bullet\Theta^{\bullet m})=(mk+1)\Theta^{\bullet m}
\end{equation}
\end{lemma}


\begin{proof}
When $m=1$ and $m=0$, Equation~\ref{eq:JKL} and Equation~\ref{eq:MNO} are
(resp.) the statements 
\begin{equation*}
\Delta(A\bullet U) =E\,\,;\,\, \Delta(W) =E.
\end{equation*}
These are well-known facts about the classes $A\bullet U$ and $W$, which are
the images of the generator of $H_{n-1}(\Omega S^{n})$. \ Also note 
\begin{equation*}
\Delta A=0
\end{equation*}
for dimension reasons, and 
\begin{equation*}
\Delta(U^{\bullet m}) =0\,\,;\,\, \Delta(\Theta^{\bullet m}) =0\,.
\end{equation*}
for all $m\geq1$:\ \ because the homology is trivial in the degree of $%
\Delta(U^{\bullet m})$ \ (or $\Delta(\Theta^{\bullet m})$ if $n$ is even) if 
$n>3$, and, in the case $n=3$ because in the standard metric $\Delta
(U^{\bullet m})$ has a representative in level $\leq\frac{m}{2}$ if $m$ is
even , and level $\leq\frac{m+1}{2}$ if $m$ is odd. But the homology gropups
in the degree of $\Delta(U^{\bullet m})$ at that level are trivial. We will
prove Equation~\ref{eq:JKL} and Equation~\ref{eq:MNO} using induction on $m$
and Chas and Sullivan's equations~\cite[Cor.5.3, p.19]{CS} or \cite[17.1.1,
p.58]{GH}. 
\begin{align}
\{X,Y\bullet Z\} & =\{X,Y\}\bullet Z+(-1)^{|Y|(|X|+1)}Y\bullet\{X,Z\}.
\label{eq:numbernumber} \\
\Delta(X\bullet Y) & =\Delta X\bullet Y+(-1)^{|X|}X\bullet\Delta
Y+(-1)^{|X|}\{X,Y\}  \label{eq:number}
\end{align}
where $|X|=\deg X-n\,.$

The first step is to compute some brackets. Using Equation~\ref{eq:number}
we have 
\begin{equation*}
\{A,U\}=(-1)^{n}\Delta\left(A\bullet U\right)=(-1)^{n}E=-E 
\end{equation*}
if $n$ is odd or $G=\mathbb{Z}_{2}$. When $n$ is even we conclude for
dimensional reasons that 
\begin{equation*}
\{W,\Theta\}=-k\Theta
\end{equation*}
for some $k=k(n)$.

Using induction and Equation~\ref{eq:numbernumber} we have 
\begin{align*}
\{A,U^{\bullet m}\} & =-mU^{\bullet(m-1)} \\
\{W,\Theta^{\bullet m}\} & =-mk\Theta^{\bullet m}
\end{align*}
for all $m\geq1$.

The lemma follows using induction and Equation~\ref{eq:number}.
\end{proof}

\textsc{Chas-Sullivan product on $H_{\ast}(\Lambda(S^{n});\mathbb{Z}_{2})$ $%
n>2$ even}

The inclusion 
\begin{equation*}
i:\Omega M\rightarrow \Lambda M
\end{equation*}%
induces maps 
\begin{align*}
i_{\ast }& :H_{\ast }(\Omega M)\rightarrow H_{\ast }(\Lambda M) \\
i^{!}& :H_{\ast }(\Lambda M)\rightarrow H_{\ast -n}(\Omega M).
\end{align*}%
Let $\times $ denote the Pontrjagin product on $H_{\ast }(\Omega M)\,.$ We
compute the homology loop product for even spheres and $\mathbb{Z}_{2}$
coefficients, Equation~\ref{eq:AXA}, using the Pontrjagin product and the
Gysin formulas~\cite[Prop.3.4]{CS}; \cite[9.3]{GH}. 
\begin{align}
i^{!}(X\bullet Y)& =i^{!}(X)\times i^{!}(Y)  \label{eq:QXQ} \\
(i_{\ast }(Z))\bullet X& =i_{\ast }(Z\times i^{!}(X))  \label{eq:RXR} \\
i_{\ast }i^{!}X& =A\bullet X  \label{eq:PXP}
\end{align}%
where $A$ is a generator of $H_{0}(\Lambda M)\,.$ The Pontrjagin ring $%
(\Omega S^{n},\mathbb{Z}_{2};\times )$ is the polynomial ring $\mathbb{Z}%
_{2}[W]$ on a generator $W$ in dimension $n-1$. It is well known (probably
since Morse, though he was more interested in the quotient 
$\Lambda /\mathbb{O}(2)$
than in $\Lambda $; see~\cite{GH} for a recent treatment using Morse theory)
that $H_{\ast }(\Lambda (S^{n});\mathbb{Z}_{2})$ has the same rank in each
dimension as $\mathbb{Z}_{2}[W]\otimes H_{\ast }(S^{n}).$ It follows that
the spectral sequence for the fibration 
\begin{equation*}
\Omega M\rightarrow \Lambda M\rightarrow M
\end{equation*}%
collapses, and thus that $i_{\ast }$ is injective, and $i^{!}$ is
surjective. In particular there is an element $U\in H_{2n-1}(\Lambda M)$
with $i^{!}(U)=W$. It follows from Equation~\ref{eq:QXQ} that $%
i^{!}U^{\bullet m}=W^{\times m}$ and thus that $U^{\bullet m}\neq 0$ for all 
$m$. By Equation~\ref{eq:PXP}, $i_{\ast }(W)=A\bullet U$. \ It follows from
Equation~\ref{eq:RXR}, and the fact that $i_{\ast }$ is injective that 
\begin{equation*}
i_{\ast }(A\bullet U)\bullet U^{\bullet m}=i_{\ast }(W\times
i^{!}U^{m})=i_{\ast }(W^{\times m})\neq 0.
\end{equation*}%
The Chas-Sullivan product is always trivial on the image of $i_{\ast }$, so $%
A\bullet A=0$. Because the rank of 
$H_{\ast }(\Lambda (S^{n});\mathbb{Z}_{2}) $ 
is at most one in each dimension, Equation~\ref{eq:AXA} follows. 
\section{Continuity Lemma and
Density Theorem
}
\label{sec:continuity}
If $M$ is a compact manifold, the space $\mathcal{G}_{1}$ is defined in
the Open Mapping theorem~\ref{thm:open_mapping_theorem} in the
introduction. \ The length, average index, index, and nullity of closed
geodesics have the following continuity properties: 
\begin{itemize}
\item[(i)] The map $\mathcal{G}_{1}\rightarrow \mathbb{R}^{2}$
 given by $(g,\gamma)\rightarrow(l_{g}(\gamma
),\alpha_{\gamma,g})$ is continuous.
\item[(ii)] The map $\mathcal{G}_{1}\rightarrow\mathbb{R}$
given by $(g,\gamma)\rightarrow \ind (\gamma)$ is \textit{lower
semi- continuous.}
\item[(iii)]
The map \ $\mathcal{G}_{1}\rightarrow\mathbb{R}$
given by $(g,\gamma)\rightarrow \index+\nullity(\gamma)$ is
\textit{upper semi- continuous.}
\end{itemize}

In particular if $(g_{q},\gamma_{q})\rightarrow(g,\gamma)$, then
$\lim\inf(\ind(\gamma_{q}))\geq \ind(\gamma)$ and $\lim\sup
(\ind+\nullity(\gamma_{q}))\leq(\ind+\nullity)(\gamma).$  

The {\em Continuity Lemma}~\ref{lem:continuity} 
is a consequence of the {\em Resonance
Theorem}~\ref{thm:resonance} and the {\em Powers Lemma}~\ref{lem:powers}: 
Fix a field $G$. For a fixed
homology or cohomology class $X$, $\cri X$ depends continuously on the metric.
By the Powers Lemma
\[
\frac{cr(\omega^{\circledast m})}{m}\leq\frac{2(n-1)}{\overline{\alpha}_{g,G}%
}\leq\frac{cr(\Theta^{\bullet m})}{m}%
\]
Since the left and right terms have the same limit as $m\rightarrow\infty$,
the Theorem follows.

\begin{proof} (of the {\em Density Theorem}~\ref{thm:density})

Fix a metric $g$ and let $\varepsilon>0$
be given. Take a nontrivial homology class 
$X\in H_{d}(\Lambda;\mathbb{Q})$ 
and approximate $g$ by a sequence of bumpy metrics $g_{i}$. Let
$\Lambda_{i}^{a}$ be the set of closed curves $\gamma$ with $F_{i}(\gamma)\leq
a,$ where $F_{i}=\sqrt{E_i}$ in the metric $g_{i}.$ Given $\mu>0$, for $i$
sufficiently large, the $d$-dimensional class $X$ will lie in the image of
$H_{d}\left(\Lambda_{i}^{\cri(X)+\mu}\right)$, and will have nontrivial image in
$H_{d}\left(\Lambda_{i}^{\cri(X)+\mu},\Lambda_{i}^{\cri(X)-\mu}\right)$. Moreover by
basic nondegenerate Morse theory, for each sufficiently large $i$, \ the
critical level $\cri_{i}X$ \ of $X$ in the metric $g_{i}$ will be equal to
$F_{i}(\gamma_{i})$ for some (nondegenerate) closed geodesic $\gamma_{i}$ with
$\ind(\gamma_{i})\leq d\leq \ind(\gamma_{i})+1.$  Note also that
\begin{equation*}
\lim_{i\to \infty}\cri_{i}(X)    =\lim_{i\to \infty}F_{i}(\gamma_{i})=\cri(X).
\end{equation*}

From Ascoli-Arzel\`a's theorem and the continuity of the map
$g \mapsto \cri_g (X)$ we conclude that 
some subsequence of $\{\gamma_{i}\}$
will converge to a closed geodesic $\gamma_{X}$ of length $\cri(X)$ with the
property that
\[
\ind(\gamma_{X})\leq \ind(\gamma_{i})\leq d\leq \ind+
\nulli(\gamma_{i})+1\leq \ind+\nulli(\gamma_{X})+1 \,.
\]

By the {\em Resonance Theorem}~\ref{thm:resonance}, if $d$ is sufficiently large, 
$\frac{d}{\cri(X)}$ is arbitrarily close to 
$\overline{\alpha}_{g,\mathbb{Q}}$. 
Using the {\em Resonant Iterates Lemma}~\ref{lem:resonant_iterates},
if $d$ is sufficiently large, we conclude 
$\overline{\alpha}_{\gamma_{X}}
\in \left(\overline{\alpha}_{g,\mathbb{Q}}-\varepsilon,
\overline{\alpha}_{g,\mathbb{Q}}+\varepsilon\right)\,.$ 
Thus if $d$ is sufficiently large, for each nontrivial
homology class $X$ \ there is a closed geodesic 
$\gamma_{X}$ of length $\cri(X)$
with $\overline{\alpha}_{\gamma_{X}}\in
(\overline{\alpha}_{g,\mathbb{Q}}-\varepsilon,\overline{\alpha}_{g,\mathbb{Q}}+\varepsilon)$.

We can assume that $\gamma_{X}$ are isolated geodesics: It follows from
(i) that if (for any $\varepsilon>0$) there were a nonisolated closed geodesic
$\gamma$ with mean frequency 
$\overline{\alpha}_{\gamma,\mathbb{Q}} \in \left(
\overline{\alpha}_{g,\mathbb{Q}}
-\varepsilon,
\overline{\alpha}_{g,\mathbb{Q}}+\varepsilon\right),$ 
then there would be infinitely many closed geodesics
$\gamma_{i}$ with mean frequency in 
$(\overline{\alpha}_{g,\mathbb{Q}}-
\varepsilon,\overline{\alpha}_{g,\mathbb{Q}}+\varepsilon).$ 
Moreover the avearage index $\alpha_{i}$ would be bounded,
and so the sum of the inverted average indices mentioned in the density
theorem would in fact be infinite.

Thus we may (for purposes of proving the Density Theorem), assume
that for sufficiently small $\varepsilon$ all closed geodesics $\gamma$ with
mean frequency 
$\overline{\alpha}_{\gamma}\in(\overline{\alpha}_{g,\mathbb{Q}}-\varepsilon,
\overline{\alpha}_{g,\mathbb{Q}}+\varepsilon)$ 
are isolated.

For each \textit{prime} closed geodesic $\gamma$ with $\overline{\alpha
}_{\gamma}\in
(\overline{\alpha}_{g,\mathbb{Q}}-\varepsilon,
\overline{\alpha}_{g,\mathbb{Q}}+\varepsilon)$,
the inverted index $\frac{1}{\overline{\alpha}_{\gamma}}$
is the number of critical levels per degree contributed by the iterates of
$\gamma$. Let $N(d)$ be the number of distinct critical levels of rational
homology classes of degree $\leq d$. It will be enough to show that
\begin{equation}
\label{eq:LJL}
\underset{d\rightarrow\infty}{\lim}\frac{N(d)}{d}  
\geq\left\{
\begin{array}{ccc}
\frac{1}{n-1}&;& n \text{ odd}\\
\frac{1}{2(n-1)} &;& n \text{  even}
\end{array}
\right.  
 \, .
\end{equation}

We use the description of the homology of $\Lambda$ given in 
Section~\ref{sec:proof_Lemma_C}.
If $n$ is even we use Equation~\ref{eq:QXR}, the 
{\em Duality Lemma}~\ref{lem:duality},
and Equation~\ref{eq:XY} to conclude that the classes $B\bullet
\Theta^{\bullet m}$ have distinct critical values. Since $\bullet\Theta$ has
degree $2(n-1),$ the theorem is proved in this case.

Now consider the case where $n$ is odd. Let $\sigma:H_{k}(\Lambda
)\rightarrow H_{k+1-n}(\Lambda)$ by $\sigma(X)=A\bullet\Delta X$.  Using
Equation~\ref{eq:JKL} we see that there is a nontrivial infinite string of rational
classes
\[
\ldots\underset{\sigma}{\longrightarrow}X_{j}
\text{ \ }\underset{\sigma}{\longrightarrow
}X_{j-1}\underset{\sigma}{\longrightarrow}\ldots\underset{\sigma}{\longrightarrow}
X_{0}=A\,.
\]
Now for all nontrivial $X$ of positive degree, $\cri(\Delta X)\leq \cri(X)$,
$\cri(A\bullet X)\leq \cri(X)$, and in fact $\cri(A\bullet X)<\cri(X)$ unless the
support of the image of $X$ in the level homology maps \textit{onto }$M$ by
the evaluation map. But this would violate the hypothesis that the geodesics
in $\Sigma_{X}$ are isolated. Therefore Equation~\ref{eq:LJL} follows. 
 The theorem
follows for $n$ odd since $\sigma$ has degree $-(n-1)$.

\end{proof}
\begin{proof}(of Corollary~\ref{cor:pinch}) Again we follow~\cite{Rad95}: Under the
given curvature assumption the length of a closed geodesic is bounded from
below by $\pi(1+\lambda^{-1})$. cf.~\cite[Thm.1]{Rad04}, which together with
the lower curvature bound shows that 
the average index satisfies
$\alpha_{\gamma}>n-1$, 
cf.~\cite[Lem.2]{Rad07}.
The statement now follows from the {\em Density
Theorem}~\ref{thm:density}.
\end{proof}
\begin{proof} (of Corollary~\ref{cor:Katok})  \, Let $n,g_{0},
\mathcal{U}, $ and $N\in\mathbb{N}$ be given. 
Pick $D\in\mathbb{N}$ with
\[
D\geq N(2n-2).
\]
Because the curvature is positive in a neighborhood of $g_{0}$, there is a
neighborhood $\mathcal{W}_{0}$ of $g_{0}$ and $L\in\mathbb{R}$
 so that for every $g\in\mathcal{W}$, all closed geodesics on
$\left(S^{n},g\right)$ of length $>L$ have average index satisfying
\[
\alpha_{\gamma}\geq D\,.
\]
There is a neighborhood $\mathcal{W}$ of $g_{0}$ in the space of metrics
so that for every $g\in\mathcal{W}$, all closed geodesics on $(S^{n},g)$
lie in $\mathcal{U}$ \textit{or} have length $>L$. 
\end{proof}
\section{Perturbation result}

\label{sec:perturbation} We can increase the mean frequency $\overline{\alpha%
}=\alpha/\ell$ of a closed geodesic $\gamma$ by increasing the average index 
$\alpha=\alpha_{\gamma}$ or by decreasing the length $\ell=\ell (\gamma).$ \
The Open Mapping Theorem~\ref{thm:open_mapping_theorem} is an immediate
consequence of the following: 

\begin{theorem}
\label{theorem:perturbation} \textrm{\textsc{(Perturbation Theorem)}} 
Given a geodesic $\gamma$ of positive length on a Riemannian or Finsler
manifold $(M,g)$, and an open neighborhood $U$ of a point $p$ on $\gamma$,
at least one of the following is true:

\begin{itemize}
\item[(i)] There is a smooth family of metrics $g^{s}$, $s\ge 0$, with $%
g^{0}=g$, and with $g^{s}=g$ outside $U$, so that $\gamma^{s}=:\gamma$
remains geodesic in the metric $g^{s}$, the average index of $\gamma^{s}$ is
increasing (resp. decreasing) for $s\geq0$ , and with the length of $%
\gamma^{s}$ constant.

\item[(ii)] There is a smooth family of metrics $g^{s,t}$, $0\leq s,t$, with 
$g^{0,0}=g$, with $g^{s,t}=g$ outside $U$, so that $\gamma ^{s,t}=:\gamma $
remains geodesic in the metric $g^{s,t}$, and so that on an open set of the
first quadrant containing the positive $s-$axis, the Poincar\'{e} map
associated to $\gamma ^{s,t}$ has no eigenvalue on the unit circle and
length $(\gamma ^{s,t})=$length$(\gamma )-t$ \ (resp. length$($ $\gamma
^{s,t})=$length$(\gamma )+t$).
\end{itemize}
\end{theorem}
Note that in the latter case the average index is constant. In either case
the mean frequency is increasing (resp. decreasing).

Let $\mathcal{M}$ be the space of real $2(n-1)$ by $2(n-1)$ matrices, and
let \\$\mathfrak{S}=\mathrm{Sp}\left( 2(n-1),\mathbb{R}\right) =\{X\in 
\mathcal{M}:X^{\ast }JX=J\},$ where 
\begin{equation*}
J=\left( 
\begin{array}{cc}
0 & -I \\ 
I & 0%
\end{array}%
\right) \in \mathfrak{S}\,.
\end{equation*}%
The Lie algebra is $\mathfrak{sp}(2(n-1))=\{A\in \mathcal{M}:A^{\ast
}J+JA=0\}\,.$ The \textit{positive cone} in the Lie algebra is $\{A\in 
\mathfrak{sp}(2(n-1)):JA>0\}$. The positive cone is invariant under
conjugation by an element of $\mathfrak{S}$. A curve $P(s)$ in $\mathfrak{S}$
is a \textit{+-curve} if the left translate of each tangent vector to the
identity lies in the +-cone, i.e. 
\begin{equation*}
JP^{-1}\frac{dP}{ds}>0
\end{equation*}%
for all $s$ . 

\begin{lemma}
\label{lem:eins} Given a metric $g_{0},$ and a closed geodesic $\gamma,$
consider a $1$-parameter perturbation $g_{s}$ of the metric on $M$ that

\begin{enumerate}
\item keeps $\gamma$ geodesic of the same length,

\item does not alter parallel transport along $\gamma,$

\item does not decrease the sectional curvature of any plane containing $%
\gamma^{\prime}$ along $\gamma,$

\item increases the sectional curvature of every plane containing $%
\gamma^{\prime}$ along $\gamma$ to first order in the perturbation parameter 
$s,$ in a neighborhood of some point on $\gamma\,.$
\end{enumerate}

\textit{Then to first order the Poincare map }$P(s)$\textit{\ moves in a
+-direction at }$s=0$\textit{, that is }%
\begin{equation*}
JP^{-1}\left.\frac{dP}{ds}\right|_{s=0}>0\,.
\end{equation*}
\end{lemma}


\begin{proof}
This follows from the same argument as Bott's proof of \cite[Prop.3.1,
p.204-205]{Bo} with minor modification. Let $\gamma $ have length $\ell $
and speed $1$. \ Pick an orthonormal basis $\mathbf{U}_{1},...\mathbf{U}%
_{n-1}$ for the subspace of $T_{\gamma (0)}M$ orthogonal to $\gamma ^{\prime
}(0)$, and use parallel transport to get a basis (with $\gamma ^{\prime }(t)$%
) $\mathbf{U}_{1}(t),...\mathbf{U}_{n-1}(t)$ for $T_{\gamma (t)}M$. \ Let $\
X_{s}(t)\in \mathfrak{S}$ be the solution of 
\begin{equation*}
\frac{\partial X_{s}(t)}{\partial t}=A_{s}(t)X_{s}(t)
\end{equation*}%
with $X_{s}(0)=I$ \ for all $s$, \ where $A$ is the matrix 
\begin{equation*}
A_{s}=\left( 
\begin{array}{cc}
0 & I \\ 
-R_{s} & 0%
\end{array}%
\right)
\end{equation*}%
with $R_{s}$ giving the sectional curvatures $\left\langle R(\gamma ^{\prime
},\bullet )\gamma ^{\prime },\bullet \right\rangle $ in terms of the
parallel translated basis. \ Suppose that parallel transport along $\gamma $
to the point $t=\ell $ results in the orthogonal symplectic transformation $%
Q $ \ on our basis, that is \ $\mathbf{U}_{i}(\ell )={\textstyle\sum }%
\mathbf{U}_{j}(0)Q_{ji}$. The Poincar\'{e} map $P(s)$ under our hypotheses
is given by 
\begin{equation*}
P(s)=\widehat{Q}X_{s}(\ell )
\end{equation*}%
where 
\begin{equation*}
\widehat{Q}=\left( 
\begin{array}{cc}
Q & 0 \\ 
0 & Q%
\end{array}%
\right) ,
\end{equation*}%
and $P(s)$ moves in a +-direction at $s=0$ if and only if $X_{s}(\ell )$
does, since 
\begin{equation*}
J(X_{s}(\ell ))^{-1}\left. \frac{\partial X_{s}(\ell )}{\partial s}%
\right\vert _{s=0}=JP^{-1}\left. \frac{dP}{ds}\right\vert _{s=0}.
\end{equation*}

By the same argument as in Bott's paper, we have (see~\cite[3.2]{Bo}): 
\begin{equation}  \label{eq:star}
J\left\{X_{s}^{-1} \left.\frac{\partial X_{s}}{\partial s}%
\right\}\right|_{t=\ell;s=0}=\underset{0}{\overset{\ell}{{\textstyle\int}}}%
X_{0}^{\ast}T(t)X_{0}(t)dt
\end{equation}
where 
\begin{equation*}
T(t)=\left( 
\begin{array}{cc}
\tau(t) & 0 \\ 
0 & 0%
\end{array}
\right)
\end{equation*}
with 
\begin{equation*}
\tau(t)=\left.\frac{\partial R_{s}(t)}{\partial s}\right|_{s=0}\geq0
\end{equation*}
and thus $\tau(t)>0$ by hypothesis for all $t$ in some open set. A nonzero
vector $v=(x,y)\in\mathbb{R}^{2(n-1)}$ defines the initial conditions of a
nontrivial solution of the Jacobi equation 
\begin{align*}
\frac{\partial x(t)}{\partial t} & =y(t) \\
\frac{\partial y(t)}{\partial t} & =-R_{0}(t)x(t)
\end{align*}
and 
\begin{equation*}
\left\langle X_{0}^{\ast}T(t)X_{0}(t)v,v\right\rangle =\left\langle
\tau(t)x(t),x(t)\right\rangle .
\end{equation*}
Since the zeroes of $x(t)$ are isolated, the quadratic form given by
Equation~\ref{eq:star} evaluated at $v$ is positive.
\end{proof}


\begin{remark} 
\label{rem:finsler} 
\rm
In this section we state the results for
Riemannian and Finsler metrics. The proofs will be carried out for
simplicity only in the Riemannian case. Here we point out the changes in the
proofs for the Finsler case: For the following results compare~\cite{Rad04}.
Assume that $f:TM\rightarrow \mathbb{R}^{\ge 0}$ is a Finsler metric on $M.$
For a nowhere vanishing vector field $V$ defined on an open subset of the
manifold we can define a Riemannian metric $g^V$ by 
\begin{equation*}
g^V(X,Y)=\frac{1}{2}\left.\frac{\partial^2 }{\partial s \partial t}%
\right|_{s=t=0} f^2\left(V+sX+tY\right)\,.
\end{equation*}
If $x=(x_1,x_2,\ldots, x_n)$ are coordinates on the manifold and $%
(x,y)=(x_1,\ldots,x_n,$ $y_1,\ldots,y_n)$ are the induced coordinates on the
tangent bundle, then the metric coefficients are given by: 
\begin{equation*}
g^V_{ij}(x,y)=\frac{1}{2}\, \frac{\partial^2}{\partial y_i \partial y_j}
f^2(x,y)\,.
\end{equation*}
A curve $x(t)$ is a geodesic if the Lagrange equations 
\begin{equation}  \label{eq:Lagrange}
\frac{d}{dt} \frac{\partial f^2}{\partial y_j}(x(t),\dot{x}(t))- \frac{%
\partial f^2}{\partial x_j}(x(t),\dot{x}(t))=0
\end{equation}
hold for $j=1,\ldots,n.$ For a geodesic $\gamma$ we take a nowhere vanishing
vector field $V$ defined in a neighborhood of the geodesic extending the
velocity vector field $\gamma^{\prime}$ along $\gamma.$ Then $\gamma$ is
also a geodesic of the Riemannian metric $g^V$ and the Jacobi fields as well
as the parallel transport along $\gamma$ with respect to the Finsler metric
and with respect to this Riemannian metric coincide. Therefore also the
linearized Poincar\'e mappings of $\gamma$ with respect to the Finsler
metric and the Riemannian metric coincide. Let $\sigma$ be a plane in the
tangent space $T_{\gamma(t)}M$ containing $\gamma^{\prime}(t)$. Then the
flag curvature $K\left(\gamma^{\prime};\sigma\right)$ of the flag $%
(\gamma^{\prime};\sigma)$ with respect to the Finsler metric $f$ coincides
with the sectional curvature $K[g^V](\sigma)$ of a two-plane $\sigma$ with
respect to the Riemannian metric $g^V.$ If we extend the following proofs to
the Finsler case we use the Riemannian metric $g^V$ and its sectional
curvature. 
\end{remark}


\begin{lemma}
\label{lem:zwei} Given a metric $g=g^{(0)}$, a closed geodesic $\gamma,$ and
a neighborhood $U$ of the point $p=\gamma(0)$ there is a one-parameter
smooth family $\overline{g}=g^{(s)}, s\in [0,\beta)$ of perturbations of the
metric $g=g^{(0)}$ supported in $U$ such that the following properties hold
for all $s \in [0,\beta):$

\begin{enumerate}
\item $\gamma$ is a closed geodesic of $\overline{g}=g^{(s)}$ of the same
length.

\item Parallel transport along $\gamma$ is unchanged.

\item Sectional curvature of any plane containing $\gamma^{\prime}$ does not
decrease, i.e the sectional curvature $\overline{K}(\sigma)=K[g^{(s)}](%
\sigma)$ with respect to the Riemannian metric $\overline{g}=g^{(s)}, s>0$
of any plane $\sigma \subset T_{\gamma(t)}M$ containing $\gamma^{\prime}(t)$
satisfies: $K [g^{(s)}](\sigma)\ge K[g](\sigma).$

\item There is $\eta>0$ such that the following holds: For all planes $%
\sigma $ containing $\gamma^{\prime}(t)$ the sectional curvature $%
K[g^{(s)}](\sigma) $ of the plane $\sigma$ with respect to the Riemannian
metric $g^{(s)}$ satisfies: 
\begin{equation*}
K[g^{(s)}](\sigma) \ge K[g](\sigma) + s
\end{equation*}
for all $t\in (-\eta,\eta)$ and $s \in[0,\beta).$
\end{enumerate}
\end{lemma}


\begin{proof}
We follow the lines of the Proof of~\cite[Prop.3.3.7]{Kl82}: Choose an
orthonormal frame $E_0,E_1,\ldots, E_{n-1}$ in $T_pM=T_{\gamma(0)}M$ with $%
E_0=\gamma^{\prime}(0)$ and extend this frame to an orthonormal and parallel
frame field $E_0(t),E_1(t),E_2(t),\ldots,E_{n-1}(t)$ (parallel with respect
to $g=g^{(0)}$) with $E_0(t)=\gamma^{\prime}(t)$ along $\gamma=\gamma(t).$
Then the mapping 
\begin{equation*}
(x_0;x_1,\ldots,x_{n-1})=(t;x_1,x_2,\ldots,x_{n-1})=(t;x) \mapsto
\exp_{\gamma(t)}\left( \sum_{i=1}^{n-1} x_i E_i(t)\right)
\end{equation*}
defines \emph{Fermi coordinates} along the geodesic $\gamma$ in a
neigbhorhood of $0 \in \mathbb{R}^n$ resp. in a neighborhood $%
U^{\prime}\subset M$ of $p=\gamma(0).$ Here $\exp$ denotes the exponential
mapping of the Riemannian metric $g.$

We denote the coordinate fields by $\partial_0=\partial / \partial
x_0=\partial /\partial t; \partial_1= \partial/\partial x_1,\ldots,
\partial_{n-1}=\partial/\partial x_{n-1}$ hence along $\gamma$ we have: $%
E_i(t)=\partial_i (t;0).$ 
Then the following equations hold for the
coefficients $g_{ij}=g\left(\partial_i,\partial_j\right)$ of the Riemannian
metric $g,$ the Christoffel symbols $\Gamma_{ij}^k$ and the coefficients $%
R_{ijkl}= g\left(R\left(\partial_i,\partial_j\right) \partial_k ,
\partial_l\right):$ 
\begin{eqnarray*}
g_{ik}(t;0)=\delta_{ik}; g_{ik,l}(t;0)=\frac{\partial g_{ik}(t;0)}{\partial
x_l}=0 \\
R_{0ik0}(t;0)= -\frac{1}{2} g_{00,ik}(t;0)= -\frac{1}{2} \frac{\partial^2
g_{00}(t;0)}{\partial x_i\partial x_k}
\end{eqnarray*}
In particular the Christoffel symbols $\Gamma_{ij}^k$ vanish along $%
\gamma=\gamma(t):$ $\Gamma_{ij}^k(t;0)=0$ for all $t$ since the coordinate
fields $\partial_1,\ldots,\partial_{n-1}$ are parallel along $\gamma.$ We
introduce the function $r:=r(x)$ by the equation $r^2=\sum_{1\le i\le n-1}
x^2_i$ and define a small perturbation $\overline{g}=g^{(s)}$ of the
Riemannian metric $g=g^{(0)}$ by: 
\begin{equation}  \label{eq:gs1}
g^{(s)}_{ij}(t;x)= \overline{g}_{ij}(t;x)=g_{ij}(t;x)-s
\delta_{i0}\delta_{j0}a(t) a(r^2) r^2\,.
\end{equation}
For some small $\eta>0$ the non-negative function $a$ satisfies: $a(t)=0$
for $|t|\ge 2 \eta$ and $a(t) =1$ for $|t|\le\eta.$ Then $\overline{g}%
_{ij}(t;0)=\delta_{ij}; \overline{g}_{ij,k}(t;0)=0$ and the Christoffel
symbols $\overline{\Gamma}_{ij}^k(t;0)$ along $\gamma$ satisfy $\overline{%
\Gamma}_{ij}^k(t;0)=\Gamma_{ij}^k(t;0)=0,$ i.e. the coordinate fields $%
\partial_1,\ldots,\partial_{n-1}$ are parallel along $\gamma$ with respect to 
$g^{(s)}.$ Therefore the curve $\gamma=\gamma(t)$ is also a geodesic of the
Riemannian metric $g^{(s)}$ and the length and the parallel transport along $%
\gamma$ is unchanged. This proves the first two claims.

Then we obtain for the Riemannian curvature tensor $\overline{R}(t;0)=%
\overline{R}_{ijkl}(t;0)$ of the Riemannian metric $g^{(s)}$ along the curve 
$\gamma=\gamma(t)$ from Equation~\ref{eq:gs1}: 
\begin{eqnarray}
\overline{R}_{i00i}(t;0)&=& \overline{R}_{i00}^i(t;0) =\overline{\Gamma}%
_{00,i}^i(t;0)- \overline{\Gamma}_{i0,0}^i(t;0) \\
&=&\frac{1}{2}\left( 2\overline{g}_{0i.0i}(t;0)-\overline{g}_{00,ii}(t;0)-%
\overline{g}_{ii,00}(t;0)\right) \\
&=& R_{i00i}(t;0)+s a(t) \,.
\end{eqnarray}
For the computation we use the standard formulas for the Christoffel
symbols: 
\begin{equation*}
\Gamma_{ij}^k= \frac{1}{2}g^{kl}\left(g_{li,j}+g_{jl,i}-g_{ij,l}\right)
\end{equation*}
and for the Riemann curvature tensor $R(\partial_i,\partial_j)%
\partial_k=R_{ijk}^l \partial_l:$ 
\begin{equation*}
R_{ijk}^l= \Gamma_{jk,i}^l-\Gamma_{ik,j}^l+ \Gamma_{jk}^r \Gamma_{ir}^l-
\Gamma_{ik}^r \Gamma_{jr}^l
\end{equation*}
with $R_{ijkl}=R_{ijk}^m g_{ml}.$ Here we use Einstein's sum convention. 
Then the sectional curvature $K[g^{(s)}]\left(
\partial_0,\partial_i\right)\left(t;0\right)= \overline{K}%
\left(\partial_0,\partial_i\right)$ of a plane generated by $%
\partial_0=\gamma^{\prime}(t)$ and $\partial_i$ is given by: $\overline{K}%
(\partial_0,\partial_i)=R_{i00i}(t;0)\,.$ This implies the third and the
fourth statement since $a(t)\ge 0$ for all $t$ and $a(t)=1$ for all $t \in
(-\eta, \eta).$
\end{proof}


\begin{remark}
\label{rem:finslerA} 
\rm
In the Finsler case the small perturbation of
the Finsler metric $f=f^0$ can be defined by as follows: 
\begin{eqnarray*}  \label{eq:finsler-perturbation}
(f^{(s)})^2\left(t,x_1,\ldots,x_{n-1},y_0,y_1,\ldots,y_{n-1}\right)= \\
f^2\left(t,x_1,\ldots,x_{n-1},y_0,y_1,\ldots,y_{n-1}\right)-s y_0^2
a(t)a(r^2)r^2
\end{eqnarray*}
Then $\gamma$ is also a geodesic of the Finsler metric $f^{(s)},$ as one can
check with the Lagrange equation~\ref{eq:Lagrange}. We let $V=\partial_0,$
and use the osculating Riemannian metric $g_s^V$ of the Finsler metric $%
f^{(s)}:$ 
\begin{equation}
(g_s^V)_{ij}= (g_0^V)_{ij}-s a\delta_{i0}\delta_{j0}a(t)a(r^2)r^2\,.
\end{equation}
Hence $g_s^V$ is of the form of the metric $\overline{g}$ given in Equation~%
\ref{eq:gs1} and we can proceed as in the above proof to obtain the result
also for the Finsler case. 
\end{remark}


\begin{lemma}
\label{lem:drei3} Given a metric $g=g^{(0)}$, a closed geodesic $\gamma$ of
length $L=L(\gamma)$ and a neighborhood $U$ of the point $p=\gamma(0)$ there
is a one-parameter smooth family $\overline{g}=g^{(s)}, s\in [0,\beta)$ of
perturbations of the metric $g=g^{(0)}$ supported in $U$ such that the
following properties hold for all $s \in [0,\beta):$

\begin{enumerate}
\item $\gamma$ is up to parametrization a closed geodesic of $\overline{g}%
=g^{(s)}$ of length $L+s.$

\item Parallel transport along $\gamma$ is unchanged.
\end{enumerate}
\end{lemma}

Let $a:\mathbb{R} \rightarrow \mathbb{R}^{\ge0}$ be a smooth function such
that $a(t)=0$ for $|t|\ge 2\eta,$ and $a(t)=1$ for all $t$ with $|t|\le \eta$
for sufficiently small $\eta>0.$ Then we define a perturbation $\overline{g}%
=g^{(s)}, s\in[0,\beta]:$ 
\begin{equation}  \label{eq:gs}
g^{(s)}_{ij}(t;x)= \overline{g}_{ij}(t;x)=g_{ij}(t;x)+s
\delta_{i0}\delta_{j0}a(t) a(r^2)\,.
\end{equation}

\begin{proof}
Along $\gamma$ we have: 
\begin{equation*}
\overline{g}_{ij}(t;0)=\delta_{ij}+s \delta_{i0}\delta_{j0}a(t)
\end{equation*}
and therefore 
\begin{equation*}
\overline{\Gamma}_{ij}^k (t;0)= \left\{ \overline{g}_{ki,j}(t;0)+ \overline{g%
}_{jk,i}(t;0)- \overline{g}_{ij,k}(t;0)\right\}/ \left(2 \overline{g}%
_{kk}(t;0)\right)=0\,.
\end{equation*}
for $(i,j,k)\not=(0,0,0).$ We conclude that $\gamma$ is up to
parametrization a geodesic for the Riemannian metric $g^{(s)}$ of length 
\begin{equation*}
L+s \int_{-2\eta}^{2\eta} \sqrt{a(t)}\,dt\ge L+ 2\eta s
\end{equation*}
and the coordinate fields $\partial_1,\ldots,\partial_{n-1}$ are parallel
along $\gamma$ with respect to the Riemannian metric $g^{(s)}.$ Hence the
parallel transport is unchanged and by changing the parameter $s$ we obtain
the claim.
\end{proof}


\begin{remark}
\rm
\label{rem:finslerB} In the Finsler case we use the following
perturbation: 
\begin{eqnarray*}
(f^{(s)})^2\left(t,x_1,\ldots,x_{n-1},y_0,y_1,\ldots,y_{n-1}\right)= \\
f^2\left(t,x_1,\ldots,x_{n-1},y_0,y_1,\ldots,y_{n-1}\right)+s y_0^2
a(t)a(r^2)
\end{eqnarray*}
Then $\gamma$ is up to parametrization also a geodesic of the Finsler metric 
$f^{(s)}$ (which follows from Equation~\ref{eq:Lagrange}) of length $\ge
L+2\eta s$ and the parallel transport along the reparametrized closed
geodesic is unchanged. 
\end{remark}


\begin{lemma}
\label{lem:drei} Fix $M$, $\gamma\in\Lambda M$, $\ell\in\mathbb{R}$, and let 
$g^{s},$ $0\leq s\leq\delta,$ be a smooth path in the space of metrics on $M$
with the property that for all $s$, $\gamma_{s}=:\gamma$ is geodesic in the
metric $g^{s}$ with length $\ell$. Assume that parallel translation along $%
\gamma_{s}$ is independent of $s$, and let $P(s)$ be the Poincar\'{e} map
(determined up to conjugation in the symplectic group) of $\gamma_{s\text{ .}%
}$ Suppose that $P(s)$ is a +-curve at $s=0.$ Then for some $\varepsilon>0$
one of the following is true:

\begin{itemize}
\item[(i)] the average index of $\gamma_{s}$ is increasing for $s\in$ $%
[0,\varepsilon)$, or

\item[(ii)] $P(s)$ has no eigenvalue on the unit circle for all $%
s\in(0,\varepsilon)$\,.
\end{itemize}
\end{lemma}


\begin{proof}
This lemma is a consequence of Bott's results. \ Let \ $\gamma $ be a closed
geodesic of length $\ell $ in the metric $g$, and, for $\lambda \in \mathbb{R%
}$ let \ $X_{\lambda }(t)\in \mathfrak{S}$ be the solution of the eigenvalue
equation 
\begin{equation*}
\frac{\partial X_{\lambda }(t)}{\partial t}=A_{\lambda }(t)X_{\lambda }(t)
\end{equation*}%
with $X_{\lambda }(0)=I$ \ for all $\lambda $, \ where $A_{\lambda }$ is the
matrix 
\begin{equation*}
A_{\lambda }=\left( 
\begin{array}{cc}
0 & I \\ 
-(R+\lambda I) & 0%
\end{array}%
\right)
\end{equation*}%
\ with $R$ giving the sectional curvature along $\gamma $ as above. \ Bott
proved that the average index $\alpha _{\gamma }$ of $\gamma $ is given by 
\begin{equation*}
\alpha _{\gamma }=\frac{1}{2\pi }{\textstyle\int }\Lambda (e^{i\theta
})d\theta
\end{equation*}%
where $\Lambda :\{z:|z|=1\}\rightarrow \mathbb{Z}$ is the intersection
number of the +-curve 
\begin{equation*}
\lambda \rightarrow X_{\lambda }(\ell )\text{; }-\infty <\lambda <0
\end{equation*}%
with the cycle 
\begin{equation*}
B_{z}:\det (X-z\widehat{Q}^{\ast })=0\,.
\end{equation*}%
Note that the multiplicity of $z$ as an eigenvalue of 
\begin{equation*}
P=\widehat{Q}X,
\end{equation*}%
is the same as the nullity of $(X-z\widehat{Q}^{\ast })$. It seems to the
authors that Bott ignored the fact that $Q$ may be nontrivial; however his
theorem is general enough to accommodate this case also.

Bott showed that the intersection number of curves with $B_{z}$ is well
defined; in particular curves whose endpoints are not on $B_{z}$, but that
are homotopic with endpoints fixed will have the same intersection number
with $B_{z}.$  Every intersection of $B_{z}$ with a +-curve is
transverse and carries the same orientation; the intersection number of a
+-curve with $B_{z}$ at a point $X$ is given by the nullity of $X-z\widehat{Q%
}^{\ast}$.

Given $(\{g_{s}\},\gamma)$, $0\leq s\leq\delta$ as above with $P^{-1}\frac {%
dP}{ds}|_{s=0}$ in the +-cone, we can assume that $P(s)$ is a +-curve for $%
s\in\lbrack0,\delta]$.\ \ For each $s$, the curve 
\begin{equation*}
\lambda\rightarrow X_{\lambda,s}(\ell)\text{; }-\infty<\lambda<0
\end{equation*}
is easily seen to be homotopic to the union of 
\begin{equation*}
\lambda\rightarrow X_{\lambda,0}(\ell)\text{; }-\infty<\lambda\leq0
\end{equation*}
with the path 
\begin{equation}  \label{eq:prozent}
\tau\rightarrow\widehat{Q}^{\ast}P(\tau),0<\tau<s\,.
\end{equation}
Let $\alpha_{s}$ be the average index of $\gamma_{s}$. \ Since the path
given in Equation~\ref{eq:prozent} is a +-curve, the intersection number $%
\Lambda_{s}(z)$ is nondecreasing in $s$ at each $z$, and $\alpha_{s}$ is
nondecreasing in $s$.

\smallskip Now let $s_{0}$, $s_{1}\in\lbrack0,\delta]$ with $s_{0}<s_{1}$. \
Assume\ the path 
\begin{equation*}
\tau\rightarrow\widehat{Q}^{\ast}P(\tau),s_{0}<\tau<s_{1}
\end{equation*}
intersects $B_{z}$ for some $z$ on the unit circle (in other words that $z$
is an eigenvalue of $P(\tau)$ for some $\tau\in(s_{0},s_{1})$).  It
follows that $\Lambda_{s_{1}}(z)$ is strictly greater than the sum of $%
\Lambda_{s_{0}}(z)$ and the nullity of $X_{0,s_{0}}-z\widehat{Q}^{\ast}$,
and thus that the same is true in some neighborhood of $z$. \ Thus $\Lambda
_{s_{1}}$ is strictly greater than $\Lambda_{s_{0}}$ on some open set
containing $z$, and $\alpha_{s_{1}}>\alpha_{s_{0}}$. \ Note that the set $%
H\subset\lbrack0,\delta]$ of $s$ for which $P(s)$ has no eigenvalue on the
unit circle is open, and that if $s\in H$, then $\alpha_{s}\in\mathbb{Z}$.
It follows that $H$ is a union of a finite number of intervals. If $0$ is
the endpoint of one of these intervals, we have (i); if not, (ii) holds.
\end{proof}

\begin{proof}
(of Theorem~\ref{theorem:perturbation})

Let a metric $g^{0}$, a geodesic $\gamma$ of length $\ell>0$ in the metric $%
g^{0}$ , and a neighborhood $U$ of $\gamma(0)$ be given. By Lemma~\ref%
{lem:eins} and Lemma~\ref{lem:zwei}, we can find $\beta>0$ and a
one-parameter smooth family $g^{s},s\in\lbrack0,\beta)$ of
perturbations of $g^{0}$, supported in $U$, and so that

\begin{enumerate}
\item $\gamma$ is a closed geodesic of $g^{s}$ of length $\ell$.

\item Parallel transport along $\gamma$ is unchanged.

\item The path of Poincar\'{e} maps $\{P(s):s$ $\in \lbrack0,\beta)\}$
is a $+$-curve.
\end{enumerate}

By Lemma~\ref{lem:drei}, (possibly after picking a new but still positive $%
\beta$) either

\begin{itemize}
\item[(i)] the average index is increasing for $s\in\lbrack0,\beta)$, or

\item[(ii)] $P(s)$ has no eigenvalue on the unit circle, for $%
s\in (0,\beta)\,.$
\end{itemize}

In the first case we are done. In the second case we apply Lemma~\ref%
{lem:drei3} to the geodesic $\gamma$ in the metric $g^{s,0}=:g^{s}$ , to get
a family $g^{s,t}$ of perturbations supported in $U$, with $%
\gamma^{s,t}=:\gamma$ geodesic of length $\ell\mp t$ in the $g^{s,t}$ metric
(Note the construction of Lemma~\ref{lem:zwei} is continuous in $s$). The average
index is an integer and thus locally constant in a neighborhood of any point 
$\left(\overline{g},\overline{\gamma }\right)\in \mathcal{G}_{1}$ where the Poincar%
\'{e} map of $\overline{\gamma}$ has no eigenvalue on the unit circle. The
theorem follows.
\end{proof}

\section{The Ellipsoid theorem}
\label{sec:ellipsoid}
Using standard comparison arguments one obtains the following estimates for 
the {\em mean frequency} $\malpha_\cg$ of a closed geodesic:
\begin{lemma} 
\label{lem:estimate}
{\rm (cf.~\cite[Rem.4.3]{Rad95},~\cite[Lem.3]{Rad04},~\cite[Lem.1]{Rad07})}
Let $\cg:S^1 \rightarrow M$ be a closed geodesic 
on an $n$-dimensional manifold $M$ with a Finsler metric
with flag curvature $K=K(\cg').$
\begin{itemize}
\item[(a)] If $K(\cg')\ge \delta^2$ for some $\delta>0$
then $\malpha_{\cg} \ge \frac{ (n-1)\delta}{\pi}.$
\item[(b)] If $K(\cg')\le \Delta^2$ for some $\Delta>0$ 
then $\malpha_{\cg} \le \frac{ (n-1)\Delta}{\pi}.$
\end{itemize}
\end{lemma}
For surfaces we show in Lemma~\ref{lem:kappa}
that one can improve these inequalities
in case of non-constant flag curvature. 
On a surface we can compute the mean frequency
of the closed geodesic $\gamma:\mathbb{R}\rightarrow M$
using conjugate points;
$$\malpha_{\gamma}= \lim_{m\to \infty}
\frac{N_m(\gamma)}{m}$$
where $N_m(\gamma)$ is the number of points $\gamma(t)$ conjugate
to $\gamma(0)$ along $\gamma\left|[0,m]\right.$. If $t_k$ is the 
parameter of the 
$k$-th conjugate point $\gamma(t_k)$ to $\gamma(0)$
then $$\malpha_{\gamma}=\lim_{m\to \infty}
\frac{k}{t_k}\,.$$
\begin{lemma}
\label{lem:conj}
Let $\gamma_j:\R \rightarrow M_j^2\,,j=1,2$ be  
two geodesics parametrized by arclength
on a surface $M_j=M_j^2$ endowed with a Finsler metric $F_j,j=1,2$
with positive flag curvature $K_j(t)=K_j(\gamma_j'(t)), j=1,2.$
We denote by $t_j, j=1,2$ 
the parameter of the first conjugate point $\cg(t_j)$
of $\gamma_j(0)$ along $\gamma_j.$

If $K_1(t)\le K_2(t)$ for all $t$
then $t_1\ge t_2$ and equality only holds if
$K_1(t)=K_2(t)$ for all 
$t \in [0,t_1].$
\end{lemma}
\begin{proof}
Denote by 
\beq
I\left[\gamma_j,T\right](y,z)=\int_0^T \left\{y'(t) z'(t) - K_j(t) y(t) z(t)\right\}\,dt
\eeq
the {\em index form} of the geodesic $\gamma_j:[0,T]\rightarrow M_j.$
We consider the index form on the space of smooth functions
$y,z:[0,T]\rightarrow \R$ with
$y(0)=z(0)=y(T)=z(T)=0.$ Assume that 
$K_1(t)\le K_2(t)$ for all $t.$

Then $t_j$ is the largest positive number $T_j$ such that
the index form $I\left[\gamma_j,T_j\right]$ is positive definite 
for $T<T_j$
and degenerate for $T=T_j.$
Choose the Jacobi field $y_1$ of the Finsler metric $F_1$ with
$y_1(0)=y_1(t_1)=0$ and $y_1'(0)=1.$
Then
\begin{eqnarray}
 I\left[\cg_2,t_1\right](y_1,y_1)&=&
 \int_0^{t_1} \left\{y_1'(t)^2-K_2(t)\, y_1^2(t)\right\}\,dt
 \nonumber
 \\
 &\le & 
 \int_0^{t_1}\left\{y_1'(t)^2-K_1(t) \, y_1^2(t)\right\}\,dt 
 =I\left[\cg_1,t_1\right](y_1,y_1)=0
 \label{eq:index}
\end{eqnarray}
Hence $t_1\ge t_2$ follows. If $t_1=t_2$
we also have equality in Equation~\ref{eq:index} which is only possible if 
$K_1(t)=K_2(t)$ for all $t\in [0,t_1]$
since $y_1(t)>0$ for all $t\in (0,t_1).$
\end{proof}
\begin{lemma}
\label{lem:kappa}
Let $\gamma_j: \R \rightarrow M_j=M_j^2, j=1,2$ be a closed geodesic 
parametrized by arc length on a surface
$M_j=M_j^2$ with Finsler metric $F_j, j=1,2$ and 
positive
flag curvature $K_j(t)=K_j(\gamma_j'(t)), j=1,2.$

Let $\ol{\alpha}_j;j=1,2$ be the {\em mean frequency} of
$\gamma_j$ with respect to the Finsler metric $F_j.$ 

If $K_1 (t)\le K_2(t)$ for all $t\in \R$ 
and if $\{t \in \R;K_1(t)=K_2(t)\}\subset \R$ is discrete then
$\ol{\alpha}_1< \ol{\alpha}_2.$ 
\end{lemma}
\begin{proof}
Let $L_j,j=1,2$ be the length of the closed geodesic $\gamma_j,$ 
of the metric $F_j,$
i.e.
$\gamma_j(t+L_j)=\cg(t)$ for all $t.$
For any $s \in [0,L_j]$  denote by $s+t_1^{(j)}(s)$ the parameter of
the first conjugate
point $\gamma_j\left(s+t_1^{(j)}(s)\right)$ of $\gamma_j(s).$
We conclude from Lemma~\ref{lem:conj}:
$t_1^{(1)}(s) > t_1^{(2)}(s)$ for all $s \in [0,L_j].$
Since the function $s \in \R\mapsto t_1^{(j)}(s)\in (0,\infty)$
is continuous and periodic there is an $\rho<1$
such that $t_1^{(1)}(s)\ge \rho^{-1} \, t_1^{(2)}(s) $ for all $s \in \R.$
Let $t_k^{(j)}(s)$ be the $k$-th conjugate point of $\gamma_j(s)$
then we conclude 
from Lemma~\ref{lem:conj}
$t_k^{(1)}(s)\ge \rho^{-1}t_k^{(2)}(s)$ and
$$\ol{\alpha}_1=\lim_{k\to \infty}\frac{k}{t_k^{(1)}(s)}
\le \rho \lim_{k\to \infty}\frac{k}{t_k^{(2)}(s)}
\le \rho \ol{\alpha}_2.$$
Since $\rho <1$ the claim follows. 
\end{proof}
\begin{lemma}
\label{lem:holonomy}
If $\cg: \mathbb{R}\rightarrow M$ is a closed geodesic on a Finsler manifold
$(M,g)$ with $\cg(t+1)=\cg(t)$ 
and $\left(e_1(t),e_2(t),\ldots,e_n(t)\right)$ is a parallel field of orthonormal
basis along $\cg$ with $e_1(t)=\cg'(t)$ and $e_i(t+1)=e_i(t)$ for all $t$
and if the sectional curvatures satisfy
$$ 0< \delta_i^2 \le K\left(\cg'(t),e_i(t)\right) \le \Delta_i^2$$
with equality only at a discrete set of parameters $t,$
then the mean frequency satisfies
$$ \frac{1}{\pi}\sum_{i=2}^n \delta_i <
 \ol{\alpha}_{\cg} <\frac{1}{\pi}\sum_{i=2}^n \Delta_i\,.$$
\end{lemma}
\begin{proof}
With respect to the parallel orthonormal basis field 
the index form splits as a sum of $(n-1)$ forms, which coincide with the
 the index form of a closed geodesic $\gamma_i(t), i=2,\ldots,n$
on a surface
with Gau{\ss} curvature $K(\gamma'(t),e_i(t)),i=2,\ldots,n$ 
along $\gamma_i.$
Therefore the mean frequency of $\gamma$ equals the sum of the
mean frequencies of the closed geodesics $\gamma_i , i=2,\ldots,n.$
Then the statement follows from Lemma~\ref{lem:estimate}.
\end{proof}

On the ellipsoid $$M=M(a_0,a_1,a_2):=\left\{(x_0,x_1,x_2)\in \R^3\,;\,
\frac{x_0^2}{a_0^2}+\frac{x_1^2}{a_1^2}+\frac{x_2^2}{a_2^2}=1\right\}$$
in $\R^3$ with three distinct principal
axis the intersections with the coordinate planes are 
ellipses which are up to parametrization simple
closed geodesics $\gamma_1,\gamma_2,\gamma_3.$ 
If $a_0=a_1<a_2$ then $\gamma_1$ is a circle and the ellipsoid
is invariant under rotation around the $x_2$-axis. Hence there is
a one-dimensional family of simple closed geodesics. In particular
the mean frequencies $\malpha_2$ and $\malpha_3$ coincide.
If $a_0<a_1=a_2$ then $\gamma_3$ is a circle and the ellipsoid is
invariant under rotation around the $x_0$-axis. In this case
there is one-dimensional family of simple closed geodesics,
in particular $\malpha_1$ and $\malpha_2$ coincide.

With the help of the preceding Lemmata we can show:
\begin{proposition}
Let $M=M(a_0,a_1,a_2)\subset \R^3$ 
be the $2$-dimensional ellipsoid 
with $0<a_0 < a_1 < a_2$ endowed with the induced Riemannian
metric.
Denote by $\gamma_{i},i=1,2,3$ the simple closed geodesic resp. the
ellipse which parametrizes the intersection of the ellipsoid with
the $(x_0,x_1), (x_0,x_2),(x_1,x_2)$-coordinate plane.
Then we obtain for the Gau{\ss} curvature
$K\left(\gamma_i(t)\right),i=1,2$ along $\gamma_i(t):$
\begin{equation*}
\frac{a_0^2}{a_1^2 a_2^2}
\le K\left(\gamma_1(t)\right)
\le \frac{a_1^2}{a_0^2 a_2^2}
\le K\left(\gamma_3(t)\right)
\le \frac{a_2^2}{a_0^2 a_1^2}
\end{equation*}
and 
\begin{equation*}
\frac{a_0^2}{a_1^2 a_2^2}
\label{eq:seccurv1}
\le K\left(\gamma_2(t)\right)
\le \frac{a_2^2}{a_0^2 a_1^2}\,.
\end{equation*}
Therefore we obtain for 
the mean frequencies
$\ol{\alpha}_{i}=\malpha_{\gamma_{i}},i=1,2,3\,:$
$$ 
\frac{a_0}{\pi a_1 a_2}
<\ol{\alpha}_1<
\frac{a_1}{\pi a_0 a_2}
<\ol{\alpha}_3<
\frac{a_2}{\pi a_0 a_1}
$$
and
$$
\frac{a_0}{\pi a_1 a_2}
<\ol{\alpha}_2<
\frac{a_2}{\pi a_0 a_1}\,.
$$
If $0<a_0=a_1<a_2$ then $\gamma_1$ is a circle and the
ellipsoid is invariant under rotation around the $x_2$-axis. Hence
there is a one-parameter family of closed geodesics generated
by rotation of $\gamma_2$ resp. $\gamma_3.$ In particular for
all these geodesics the mean frequency coincides and:
\begin{equation}
\label{eq:equalaxis1}
\frac{1}{\pi a_2}
=\ol{\alpha}_1<\malpha_2=\malpha_3<
\frac{a_2}{\pi a_0^2}\,.
\end{equation}
If $0<a_0<a_1=a_2$ then $\gamma_3$ is a circle and the
ellipsoid is invariant under rotation around the $x_0$-axis.
Hence by rotation of $\gamma_0$ we obtain a one-parameter
family of closed geodesics and:
\begin{equation}
\label{eq:equalaxis2}
\frac{a_0}{\pi a_1^2}
<\malpha_1=\malpha_2
<\malpha_3=
\frac{1}{\pi a_0}\,.
\end{equation}
\end{proposition}
\begin{proof}
The ellipsoid $M(a_0,a_1,a_2)$
can be parametrized as follows:
\begin{equation*}
 x_0=a_0\cos u \sin v\,;\,
x_1=a_1\sin u \sin v\,;\,
x_2=a_2 \cos v\,.
\end{equation*}
Then the Gau{\ss} curvature is given by
\begin{equation}
 K(u,v)=\frac{a_0^2 a_1^2 a_2^2}{\left\{a_0^2a_1^2 \cos^2 v+
a_2^2\left(a_1^2\cos^2u + a_0^2\sin^2 u\right) \sin^2v\right\}^2}\,,
\end{equation}
cf.~\cite[\S13]{Gray}.
The Gau{\ss} curvature along the ellipse 
$\gamma_{1}(t)=\left(a_0 \cos t ,a_1 \sin t, 0\right)$
is given by:
\begin{equation}
 K\left(\gamma_1(t)\right)= \frac{1}{a_2^2} \,
\frac{a_0^2 a_1^2}{\left(a_1^2\cos^2t + a_0^2\sin^2 t\right)^2}\,,
\end{equation}
hence
\begin{equation}
\label{eq:kc}
\frac{a_0^2}{a_1^2 a_2^2} \le K\left(\gamma_1(t)\right)
\le \frac{a_1^2}{a_0^2 a_2^2}\,.
\end{equation}
Then the statements follow from Lemma~\ref{lem:estimate}
and Lemma~\ref{lem:kappa}, the 
estimate for the mean frequencies $\malpha_2,\ol{\alpha}_3$
of the ellipses $\gamma_2,\gamma_3$ work
 analogously.
\end{proof}
With the help of this Lemma we can present the
\begin{proof}({\em of Theorem}~\ref{thm:ellipsoid}(a)):
For smooth functions $a_i(s), i=0,1,2, s\in [0,1]$
with $0<a_0(s)\le a_1(s)\le a_2(s),$
and $a_0(0)<a_1(0)=a_2(0)$ and $a_0(1)=a_2(1)<a_2(1)$
we obtain a smooth family of ellipsoids 

$M\left(a_0(s),a_1(s),a_2(s)\right)$
and we denote by $\malpha_i(s),i=1,2,3, s\in[0,1]$
the mean frequency of the simple closed geodesic $\gamma_i.$
Then Equation~\ref{eq:equalaxis1} and
Equation~\ref{eq:equalaxis2}
imply
$\malpha_1(0)=\malpha_2(0)<\malpha_3(0)$
and $\malpha_1(1)<\malpha_2(1)=\malpha_3(1).$
This implies that there is an non-empty open subset
$I\subset (0,1)$ such that 
$a_0(s)<a_1(s)<a_2(s)$ for all $s\in I.$
\end{proof}
Theorem~\ref{thm:ellipsoid}(b)
is a consequence of the following
\begin{proposition}
Let $n=2m,$ resp. $n=2m-1, m\ge 2,$ and choose $\mu>1.$
Let $M=E\left(\mu,\lambda\right)=
M\left(a_0,a_1,\ldots,a_n\right):=\left\{(x_0,x_1,\ldots,x_n)\in \R^n\,;\,
\sum_{i=0}^n \left(x_i/a_i\right)^2=1\right\}$
be the $n$-dimensional ellipsoid 
with 
pairwise distinct principal axis
$a_{2i}=\mu^i, a_{2i+1}=\lambda \mu^i, i\le m,$  
for some $\lambda>1$ endowed with the induced Riemannian
metric.
Denote by $\gamma_{i+1}=\gamma_{(2i,2i+1)},0\le i\le m$ the simple closed geodesic resp. the
ellipse which parametrizes the intersection of the ellipsoid with
the $\left(x_{2i},x_{2i+1}\right)$-coordinate plane.
Then for sufficiently small $\lambda>1$ the mean
frequencies $\malpha_i$ of the ellipses $\gamma_i,i=1,2,\ldots,m$
satisfy:
$$ 
\ol{\alpha}_1
<\ol{\alpha}_2<\cdots
<\ol{\alpha}_m\,.
$$
\end{proposition}
\begin{proof}
We give the proof in the case $n=2m:$
If $e_0,e_1,\ldots,e_n$ is the standard basis of $\mathbb{R}^{n+1},$
let $V_{i,j}$ for $j\not\in\{2i, 2i+1\}$ be the three dimensional subspace
$V_{i,j}=\left\{ x_{2i} e_{2i}+x_{2i+1}e_{2j+1}+x_je_j\,;\, x_{2i}, x_{2i+1}, x_j\in
\mathbb{R}\right\}$ of $\mathbb{R}^{n+1}.$ This space is the fixed point set of
the reflection $R: \mathbb{R}^{n+1}\rightarrow \mathbb{R}^{n+1}$
with $R(x_k)=-x_k, k\in\{2i,2i+1,j\}$ and $R(x_k)=x_k$ otherwise. 
Therefore the two-dimensional ellipsoid
$V_{i,j}\cap E\left(\mu,\lambda\right) $ is a totally geodesic submanifold.
The Gau{\ss} curvature along the ellipse $\gamma_i$ as a geodesic on this
two-dimensional ellipsoid satisfies (cf. Equation~\ref{eq:kc}):
\begin{equation}
\frac{1}{a_j^2} \, \frac{1}{\lambda^2}\le K\left(\gamma_i(t)\right)
\le \frac{1}{a_j^2} \,  \lambda^2\,;\, j \not\in\{2i,2i+1\}
 \end{equation}
with equality only at a discrete set of parameters.

From Lemma~\ref{lem:holonomy} we conclude
for the mean frequency $\ol{\alpha}_i$ of $\gamma_i:$
\begin{equation}
 \frac{1+\lambda}{\lambda}\sum_{0\le k\le m; k\not=i} \mu^{-k} < \malpha_i
<(1+\lambda)\,\lambda\, \sum_{0\le k\le m; k\not=i} \mu^{-k}\,.
\end{equation}
It follows from this estimate that for 
$1 <\lambda< 1+(\mu-1)^2\mu^{-m-2}$ 
the inequality
$$\malpha_{i}<\malpha_{i+1}$$
is satisfied for $i=1,2,\ldots,m.$
\end{proof}

\section{Appendix A}
\label{sec:more}
In this appendix we give proofs for the following estimates for
the difference $\deg(X)-\malpha \,\cri(X)$:
\begin{lemma}
\label{lem:resonant_iterates} 
{\rm \sc (Resonant Iterates)}\,
Let $\group=\mathbb{O}(2)$ or $S\mathbb{O}(2)$.
If $X\in H_{\ast}(\Lambda)$ lies hanging on a closed
geodesic $\gamma$ of length $L=\cri(X)$
in the sense that $X$
is in the image a class in of 
$H_{\ast}(\Lambda^{<L}\cup\group\cdot\gamma)$ whose image
in 
$H_{\ast}(\Lambda^{<L}\cup \group\cdot\gamma\,,\,\Lambda^{<L})$
is nontrivial (cf. Equation~\ref{eq:local_homology}), then 
\begin{equation}
-(n-1)\leq\mathrm{\deg}(X)-\overline{\alpha}(\gamma)\,\cri (X)\leq n\,.
\end{equation}
\end{lemma}
\begin{proof}
By standard Morse theory estimates
\begin{equation}
\mathrm{ind}(\gamma)\leq\mathrm{deg}(X)\leq\mathrm{ind}(\gamma)+\mathrm{null}%
(\gamma)+1
\end{equation}
and $\mathrm{null}(\gamma)\leq2(n-1)$. We obtain the lemma using these and
the estimate

\begin{equation}
\label{eq:Lalpha}
L\overline{\alpha}(\gamma)-(n-1)\leq\mathrm{ind}(\gamma)\leq L\overline
{\alpha}(\gamma)+(n-1)-\mathrm{null}(\gamma)
\end{equation}
(cf.\cite[Theorem 10.1.2]{Lo})
\end{proof}
\begin{lemma}
\label{lem:X}
If the metric $g$ carries only finitely many
closed geodesics, then for $\deg(X)$ sufficiently
large, each homology class $X$  satisfies $\left|\deg(X)-\overline{\alpha
}\,\cri(X)\right|\leq n$.
\end{lemma}
\begin{proof}
By the {\em Resonant Iterates
Lemma}~\ref{lem:resonant_iterates}, 
for each $X$ there is a closed geodesic $\gamma$ of
mean frequency $\overline{\alpha}(\gamma)$ so that the point $(\cri(X),\deg X)$
lies at a vertical distance at most $n$ from the line
\[
d=\overline{\alpha}(\gamma)\ell.
\]
On the other hand by the {\em Resonance Theorem}~\ref{thm:resonance}, the point 
$\left(\cri(X),\deg X\right)$ lies at a bounded distance from the line
\[
d=\overline{\alpha}_{g}\ell.
\]
where $\overline{\alpha}_{g}$ is the global mean frequency. These can both be
true for $\deg X$ large only if the slopes are equal: $\overline{\alpha
}(\gamma)=\overline{\alpha}_{g}.$  Lemma~\ref{lem:X} then follows from another
application of the {\em Resonant Iterates Lemma}~\ref{lem:resonant_iterates}
\end{proof}
\section{Appendix B}
\label{sec:appendixB}
In this appendix we show the 
vanishing of the string bracket for spheres and rational Coefficients,
see \cite[p.186-187]{GH}. Let $M$ be compact and orientable of
dimension $n$. As in~\cite{CS}, one may consider the 
$T=S^{1}$-equivariant homology $H_{\ast }^{T}(\Lambda )$ of the free loop
space $\Lambda (M).$ Let $\rm{ET}\rightarrow \rm{BT}$ be the classifying space and
universal bundle for $T=S^{1}$; and let $\pi :\Lambda \times \rm{ET}\rightarrow
\Lambda _{T}=\Lambda \times _{T}\rm{ET}$ be the Borel construction.

There are maps back and forth
\begin{eqnarray*}
\pi _{\ast } &:&H_{k}(\Lambda )\rightarrow H_{k}^{T}(\Lambda ) \\
\pi ^{!} &:&H_{k}^{T}(\Lambda )\rightarrow H_{k+1}(\Lambda ) \\
\pi ^{\ast } &:&H_{T}^{k}(\Lambda ,\Lambda ^{0})\rightarrow H^{k}(\Lambda
,\Lambda ^{0}) \\
\pi _{!} &:&H^{k}(\Lambda ,\Lambda ^{0})\rightarrow H_{T}^{k-1}(\Lambda
,\Lambda ^{0})
\end{eqnarray*}%
We will use rational coefficients.  The Chas-Sullivan {\em
string bracket} product on equivariant homology is defined
using the Chas Sullivan product $\bullet $: \ if 
$X,Y\in H_{\ast}^{T}(\Lambda )$, then 
\beq
\label{eq:bracket}
\lbrack X,Y\rbrack=(-1)^{|X|}\pi _{\ast }\left( \pi ^{!}(X)\bullet \pi
^{!}(Y)\right) .
\eeq
Here $|X|=i-n$ if  $X\in H_{i}^{T}(\Lambda )$. \ The action of $T=S^{1}$
preserves the energy function, so the (homology) string bracket  satisfies
the same energy estimates~\ref{eq:XY} as the Chas Sullivan
product.

Similarly the (cohomology) product $\circledast $ gives rise to a product in
equivariant cohomology: 
\beq
\label{eq:circled}
x\circledcirc y=(-1)^{|x|}\pi _{!}\left( \pi ^{\ast }(x)\circledast \pi
^{\ast }(y)\right) .
\eeq
satisfying energy estimates as Equations~\ref{eq:XY}, with $|x|=i+n-1$ in
degree $i$. 
\begin{proposition}
\label{pro:vanishing}
For an $n$-dimensional sphere $M$ and for rational coefficients the string product
$\lbrack.,.\rbrack$ on equivariant homology, cf. Equation~\ref{eq:bracket}, and the
product $\circledcirc$ on equivariant cohomology, cf. Equation~\ref{eq:circled}
both vanish.
\end{proposition}
\begin{proof}
 First cohomology: By~\cite[4.2, p.104]{H} if $n$ is
even, $H_{T}^{k}(\Lambda ,\Lambda ^{0};\mathbb{Q})=0$ 
unless $k$ is odd;  if $n$ is odd, 
$H_{T}^{k}(\Lambda ,\Lambda ^{0};\mathbb{Q})=0$
unless $k$ is even. It follows from this and the definition (keeping
track of the degrees) that the string cohomology bracket is trivial. Next
homology: By if $n$ is odd, ~\cite[4.2, p.105]{H} 
$H_{k}^{T}(\Lambda, \mathbb{Q})=0$ 
unless $k$ is even and an argument similar to the above argument for
cohomology shows that the string homology bracket is trivial. 
According~\cite[p.143]{H2},
if $n$ is even, and $X\in H_{k}^{T}(\Lambda ;\mathbb{Q})$
is nontrivial, then  $k$ is odd \textit{or }$X$ is in the image of $%
H_{k}^{T}(\divideontimes ;\mathbb{Q})$ 
for a basepoint $\divideontimes \in \Lambda ^{0}$. \ But\ $\pi ^{!}$ is
trivial on  the image of $H_{\ast }^{T}(\divideontimes ;\mathbb{Q})$ 
(since $\pi ^{!}$ increases the degree by $1$, and $H_{\ast
}^{T}(\divideontimes ;
\mathbb{Q})\neq 0$ only in even degrees). Thus if $X$ is of even degree,  $[X,Y]=0$
for any $Y$.  The rest of the argument is the same as before. 
\end{proof}

\end{document}